\documentclass{article}
\usepackage{amssymb}
\usepackage{amsfonts}
\usepackage{amsmath}
\usepackage{geometry}
\usepackage{natbib}
\usepackage{latexsym}
\usepackage{theorem}
\usepackage{graphicx}

\newtheorem{theorem}{Theorem}

\newtheorem{corollary}[theorem]{Corollary}

{\theorembodyfont{\normalfont}
\newtheorem{remark}[theorem]{Remark}
\newtheorem{example}[theorem]{Example}
}

\newtheorem{lemma}[theorem]{Lemma}

\newtheorem{proposition}[theorem]{Proposition}

\newenvironment{proof}[1][Proof]{\noindent\textbf{#1.} }{\hfill\ \rule{0.5em}{0.5em}}
\begin{document}

\title{On the class of distributions of subordinated L\'{e}vy processes}
\author{Orimar Sauri \\
%EndAName
Department of Economics and CREATES\\
Aarhus University\\
osauri@creates.au.dk \and Almut E.~D.~Veraart \\
%EndAName
Department of Mathematics\\
Imperial College London\\
a.veraart@imperial.ac.uk}
\date{\today }
\maketitle

\begin{abstract}
This article study the class of distributions obtained by subordinating L%
\'{e}vy processes and L\'{e}vy bases. To do this we derive properties of a
suitable mapping obtained via L\'{e}vy mixing. We show that our results can
be used to solve the so-called recovery problem for general L\'{e}vy bases
as well as for moving average processes which are driven by subordinated L\'{e}%
vy processes.
\end{abstract}


%\tableofcontents
%\newpage

\section{Introduction}

L\'{e}vy processes and L\'{e}vy bases constitute important building blocks
for constructing realistic models for temporal and or spatial phenomena. In
addition, it has often been noted that \emph{stochastic
volatility/intermittency}, which can be regarded as stochastic variability
beyond the fluctuations described by the L\'{e}vy noise, is present in
empirical data of e.g.~asset prices in finance or turbulence in physics. One
possibility of accounting for such additional variability is by using the
concept of subordination, see \cite{Clark1973}, \cite{Monroe1978} and also 
\cite{VeraartWinkel2010} for additional references. In this article, we are
interested in characterizing the law of subordinated L\'{e}vy processes and L%
\'{e}vy bases. We will approach this problem from the general angle of
defining a mapping 
\begin{equation}
\Phi _{M}\left( \rho \right) \left( A\right) :=\int_{0}^{\infty }\mu \left(
s,A\right) \rho \left( ds\right) ,\text{ \ \ }A\in \mathcal{B}\left( \mathbb{%
R}\backslash \left\{ 0\right\} \right) ,  \label{mapping}
\end{equation}%
where $M=\left( \mu \left( s,\cdot \right) \right) _{s\geq 0}$ denotes a
measurable collection of probability measures and $\rho $ is a L\'{e}vy
measure on $\mathbb{R}^{+}$. We note that if $\mu \left( s,\cdot \right) $
is chosen as the law of a L\'{e}vy process at time $s$ and $\rho $ is a $%
\sigma $-finite measure supported on the positive half line, then $\Phi _{M}$
describes the L\'{e}vy measure of a subordinated L\'{e}vy process. We will
discuss in which sense the mapping is related to the concept of \emph{L\'{e}%
vy mixing}, as introduced by \cite{BNPabreuThorbS2013}.

While the mapping \eqref{mapping} can be defined for various measures $\mu $%
, we are mainly interested in the situation when $\mu (s,dx)=\mu ^{s}(dx)$
for an infinitely divisible (ID) law $\mu $. In that case, we shorten the
notation and typically write $\Phi _{\mu }=\Phi _{M}$. In particular, we
focus on three scenarios in more detail: The cases when 1.) $\Phi _{\mu }$
is restricted to the finite L\'{e}vy measures on $(0,\infty )$, 2.) $\mu $
is symmetric, and 3.) $\mu $ is supported on $\left( 0,\infty \right)$. Our
first results in this context contain a detailed description of the
properties of the mapping $\Phi _{\mu }$. In particular, we characterize its
L\'{e}vy domain and some properties of its range as well as establish
conditions under which the mapping $\Phi _{\mu }$ is one-to-one.

These results can then be used to describe the law of subordinated L\'{e}vy
processes and bases.

As an application of our results, we study the so-called \emph{recovery
problem} for subordinated ID processes: If we observe a subordinated L\'{e}%
vy process $X_t=L_{T_t}$ (where $L$ is a L\'{e}vy process and $T$ is an
independent subordinator), can we recover $T$ from $X$ and, if so, in which
sense? In answering this question we will build upon and further extend
earlier work by \cite{WinMaT01} and \cite{GemMAdYor02}.

Moreover, we will go one step further and use such a subordinated L\'{e}vy
process as the driving noise in a moving average type process and derive
suitable conditions which allow us to recover the subordinator from
observations of the moving average process. In the special case of a
Brownian moving average process restricted to be a semimartingale, such a
problem has been addressed by \cite{ComteRenault1996}. More recently, \cite%
{Sau14} studied the invertibility of infinitely divisible continuous moving
averages processes and we can build on this result to solve the recovery
problem in the more general set-up which includes non-semimartingale
processes.

The outline for the remaining article is as follows. Section 2 introduces
the basic notation and background material on L\'{e}vy processes, L\'{e}vy
bases, subordination and meta-times. In Section 3, we will define the
mapping \eqref{mapping} and study its key properties. We will then use these
results in Section 4 to describe the law of subordinated L\'{e}vy processes
and bases. The recovery problem for moving average processes driven by
subordinated L\'{e}vy processes is then studied in Section 5.

\section{Preliminaries and basic results}

Throughout this article $\left( \Omega ,\mathcal{F},\left( \mathcal{F}%
_{t}\right) _{t\in \mathbb{R}},\mathbb{P}\right) $ denotes a filtered
probability space satisfying the usual conditions of right-continuity and
completeness. A two-sided L\'{e}vy process $\left( L_{t}\right) _{t\in 
\mathbb{R}}$ on $\left( \Omega ,\mathcal{F},\mathbb{P}\right) $ is a
real-valued stochastic process taking values in $\mathbb{R}$ with
independent and stationary increments whose sample paths are almost surely c%
\`{a}dl\`{a}g. If we drop the stationarity in the increments of $\left(
L_{t}\right) _{t\in \mathbb{R}}$, we call it an additive processes. We say
that $\left( L_{t}\right) _{t\in \mathbb{R}}$ is an $\left( \mathcal{F}%
_{t}\right) $-L\'{e}vy process if for all $t>s,$ $L_{t}-L_{s}$ is $\mathcal{F%
}_{t}$-measurable and independent of $\mathcal{F}_{s}$.

Denote by $ID\left( \mathbb{R}\right) $ the space of probability measures on 
$\mathbb{R}$ that are infinitely divisible. The law of a L\'{e}vy process is
in a bijection with $ID\left( \mathbb{R}\right) $ and $L_{1}$ has a L\'{e}%
vy-Khintchine representation given by 
\begin{equation*}
\log \widehat{\mu }\left( z\right) =iz\gamma -\frac{1}{2}z^{2}b+\int_{%
\mathbb{R}^{d}}\left[ e^{izx}-1-i\tau \left( x\right) z\right] \nu \left( 
\mathrm{d}x\right) ,\text{ \ \ }z\in \mathbb{R}\text{,}
\end{equation*}%
where $\widehat{\mu }$ is the characteristic function of the distribution of 
$L_{1}$, $\gamma \in \mathbb{R}$, $b\geq 0$ and $\nu $ is a L\'{e}vy
measure, i.e. $\nu \left( \left\{ 0\right\} \right) =0$ and $\int_{\mathbb{R}%
} (1\wedge \left\vert x\right\vert ^{2}) \nu \left( \mathrm{d}x\right)
<\infty .$ Here, we assume that the truncation function $\tau $ is given by $%
\tau \left( x\right) =x\mathbf{1}_{\left\{ \left\vert x\right\vert \leq
1\right\} },\ $for $x\in \mathbb{R}$. When $\mu $ has support on $\left(
0,\infty \right) $, we write $\mu \in ID\left( \mathbb{R}^{+}\right) $. We
say that $\mu \in ID\left( \mathbb{R}\right) $ is strictly $\alpha $-stable
if for any $s>0$ and $z\in \mathbb{R}$, $\widehat{\mu }\left( z\right) ^{s}=%
\widehat{\mu }\left( s^{1/\alpha }z\right) $. If $\mu \in ID\left( \mathbb{R}%
\right) $ has characteristic triplet $\left( \gamma ,b,\nu \right) $
satisfying $\int_{\mathbb{R}}\left( 1\wedge \left\vert x\right\vert \right)
\nu \left( \mathrm{d}x\right) <\infty $, we use the triplet relative to the
truncation function $\tau _{0}\left( x\right) =0,$ and we write $\left(
\beta _{0},b,\nu \right) _{0}$. Here $\beta _{0}=\gamma -\int_{\left\vert
x\right\vert \leq 1}x\nu \left( \mathrm{d}x\right) $. Moreover, in the same
context, if $b=0$, we will write $\left( \beta _{0},\nu \right) _{0}$
instead of $\left( \beta _{0},0,\nu \right) _{0}$ and we refer to it as the
characteristic pair.

By \emph{infinitely divisible continuous time moving average} (IDCMA)
process we mean a real-valued stochastic process $\left( X_{t}\right) _{t\in 
\mathbb{R}}$ on $\left( \Omega ,\mathcal{F},\left( \mathcal{F}_{t}\right)
_{t\in \mathbb{R}},\mathbb{P}\right) $ given by the following formula%
\begin{equation}
X_{t}:=\theta +\int_{\mathbb{R}}f\left( t-s\right) dL_{s},\text{ \ \ }t\in 
\mathbb{R}\text{.}  \label{eqn1.1}
\end{equation}%
where $\theta \in \mathbb{R}$, $f$ is a deterministic function and $L$ is a L%
\'{e}vy process with triplet $\left( \gamma ,b,\nu \right) $. Observe that $%
X $ is strictly stationary and infinitely divisible in the sense of \cite%
{idprocessesmaejimasatoole} and \cite{sdfields}.

A \textit{L\'{e}vy semistationary process} ($\mathcal{LSS}$) on $\mathbb{R}$
is a stochastic process $\left( Y_{t}\right) _{t\in \mathbb{R}}$ on $\left(
\Omega ,\mathcal{F},\left( \mathcal{F}_{t}\right) _{t\in \mathbb{R}},\mathbb{%
P}\right) $ which is described by the following dynamics%
\begin{equation}
Y_{t}=\theta +\int_{-\infty }^{t}g\left( t-s\right) \sigma
_{s}dL_{s}+\int_{-\infty }^{t}q\left( t-s\right) a_{s}ds,\text{ \ \ }t\in 
\mathbb{R}\text{,}  \label{eqn1.2}
\end{equation}%
where $\theta \in \mathbb{R}$, $L$ is a L\'{e}vy process with triplet $%
\left( \gamma ,b,\nu \right) $, $g$ and $q$ are deterministic functions such
that $g\left( x\right) =q\left( x\right) =0$ for $x\leq 0$, and $\sigma $
and $a$ are adapted c\`{a}dl\`{a}g processes with $\sigma $ predictable.
When $L$ is a two-sided Brownian motion, $Y$ is called a \textit{Brownian
semistationary process }($\mathcal{BSS}$). For further references to the
theory and applications of L\'{e}vy semistationary processes, see for
instance \citet{dayaheadalmut} and \citet{aproximatinglss}. See also %
\citet{RePEc:bla:jtsera:v:34:y:2013:i:3:p:385-404}. In the absence of a
drift and a stochastic volatility component, an $\mathcal{LSS}$ process is
an IDCMA process.

\subsection{L\'{e}vy basis}

In this part we present the results of a time change of a L\'{e}vy basis
that are going to be used to study the marginal distribution of a
subordinated L\'{e}vy basis. We refer to \cite{BNJP(12)} for more details on
the background material.

Let $\mathcal{S}$ be a non-empty set and $\mathcal{R}$ a $\delta $-ring of
subsets of $\mathcal{S}$ having an increasing sequence $\left\{
S_{n}\right\} \subset \mathcal{S}$ converging to $\mathcal{S}$. A stochastic
field $L=\left\{ L\left( A\right) :A\in \mathcal{R}\right\} $ defined on $%
\left( \Omega ,\mathcal{F},\mathbb{P}\right) $ is called \emph{independently
scattered random measure} (i.s.r.m. for short), if for every sequence $%
\left\{ A_{n}\right\} _{n\geq 1}$ of disjoint sets in $\mathcal{R}$, the
random variables $\left( L\left( A_{n}\right) \right) _{n\geq 1}$ are
independent, and if $\cup _{n\geq 1}A_{n}$ belongs to $\mathcal{R}$, then we
also have%
\begin{equation*}
L\left( \bigcup_{n\geq 1}A_{n}\right) =\sum_{n\geq 1}L\left( A_{n}\right) ,%
\text{ \ \ a.s.,}
\end{equation*}%
where the series is assumed to converge almost surely. A L\'{e}vy basis is
an i.s.r.m. for which the law of $L\left( A\right) $ belongs to $ID\left( 
\mathbb{R}\right) $ for any $A\in \mathcal{R}$. Any L\'{e}vy basis admits a L%
\'{e}vy-Khintchine representation: 
\begin{equation*}
\mathrm{C}\{z\ddagger L\left( A\right) \}=\int_{A}\psi \left( z,s\right)
c\left( \mathrm{d}s\right) \text{, \ \ }z\in \mathbb{R},A\in \mathcal{R},
\end{equation*}%
where $\mathrm{C}\{\theta \ddagger X\}$ denotes the cumulant function of a
random variable $X$ and%
\begin{equation}
\psi \left( z,s\right) :=iz\gamma \left( s\right) -\frac{1}{2}b\left(
s\right) z^{2}+\int_{\mathbb{R}}\left[ e^{izx}-1-i\tau \left( x\right) z%
\right] \nu \left( s,\mathrm{d}x\right) ,\text{ \ \ }z\in \mathbb{R},s\in 
\mathcal{S}\text{.}  \label{eqn1.3}
\end{equation}%
The functions $\gamma, b$ and $\nu \left( \cdot ,dx\right) $ are $\left( 
\mathcal{S},\mathcal{B}_{\mathcal{S}}\right) /\left( \mathbb{R},\mathcal{B}%
\left( \mathbb{R}\right) \right) $-measurable (in which $\mathcal{B}_{%
\mathcal{S}}:=\sigma \left( \mathcal{R}\right) $) with $b$ being
non-negative and $\nu \left( s,\cdot \right) $ is a L\'{e}vy measure for
every $s\in \mathcal{S}$. The measure $c$ is defined on $\left( \mathcal{S},%
\mathcal{B}_{\mathcal{S}}\right) $ and is called the \emph{control measure}
of $L$. The quadruplet $\left( \gamma, b,\nu, c\right) $ characterizes
uniquely the distribution of $L$ and for this reason it will be called the 
\emph{characteristic quadruplet} of $L$. Define $L^{\prime }=\left(
L^{\prime }\left( s\right) \right) _{s\in \mathcal{S}}$ the collection of
infinitely divisible random variables in $\mathbb{R}$ for which $L^{\prime
}\left( s\right) $ has characteristic triplet $\left( \gamma \left( s\right)
,b\left( s\right) ,\nu \left( s,\cdot \right) \right) $ for any $s\in 
\mathcal{S}$. The collection $L^{\prime }$ is known as L\'{e}vy seeds and
they satisfy that $\psi \left( z,s\right) =\mathrm{C}\{z\ddagger L^{\prime
}\left( s\right) \}$ for all $s\in \mathcal{S}$.

When $\gamma ,b$ and $\nu \left( \cdot ,dx\right) $ do not depend on $s$ and 
$c$ is the Lebesgue measure (up to a constant), $L$ is termed as \emph{%
homogeneous}. In this case we say that $L$ has triplet $\left( \gamma ,b,\nu
\right) .$ In the case that $\mathcal{S=}\mathbb{R}^{k}$, homogeneity is
equivalent to saying that $L\left( A+x\right) \overset{d}{=}L\left( A\right) 
$ for any $x\in \mathbb{R}^{k}$. Note that when $L$ is homogeneous in $%
\mathbb{R}^{k}$ with triplet $\left( \gamma ,b,\nu \right) $, then%
\begin{equation}
\mathcal{L}\left[ L\left( A\right) \right] =\mu ^{Leb\left( A\right) }\text{%
, \ \ }A\in \mathcal{B}_{b}\left( \mathbb{R}^{k}\right) ,  \label{eqn1.4}
\end{equation}%
where $\mathcal{B}_{b}\left( \mathbb{R}^{k}\right) $ denotes the bounded
Borel sets on $\mathbb{R}^{k}$ and $\mu $ is the infinite divisible law
associated with the triplet $\left( \gamma ,b,\nu \right) $.

If we add the extra condition that $L\left( \left\{ x\right\} \right) =0$
a.s.~for all $x\in \mathbb{R}^{k}$, then $L$ has a L\'{e}vy-It\^{o}
decomposition: We have that almost surely%
\begin{equation}
L\left( A\right) =\int_{A}\gamma \left( s\right) c\left( \mathrm{d}s\right)
+W\left( A\right) +\int_{A}\int_{\left\vert x\right\vert >1}xN\left(
dxds\right) +\int_{A}\int_{\left\vert x\right\vert \leq 1}x\widetilde{N}%
\left( dxds\right) \text{, \ \ }A\in \mathcal{B}_{b}\left( \mathbb{R}%
^{k}\right) ,  \label{eqn1.5}
\end{equation}%
where $W$ is a centered Gaussian process with $\mathbb{E}\left[ W\left(
A\right) W\left( B\right) \right] =\int_{A\cap B}b\left( s\right) c\left( 
\mathrm{d}s\right) $ for all $A,B\in \mathcal{R}$, $\widetilde{N}$ and $N$
are compensated and non-compensated Poisson random measures on $\mathbb{R}%
^{k}\times \mathbb{R}$ with intensity $\nu \left( s,\mathrm{d}x\right)
c\left( \mathrm{d}s\right) $, respectively. Additionally, $W$ and $N$ are
independent. See \cite{LevyItodecompositionrm} for more details.

In the homogeneous case, when $\mathcal{S}=\mathbb{R}$, the process $%
L_{t}-L_{s}:=L\left( \left( s,t\right] \right) $ is an $\left( \mathcal{F}%
_{t}\right) $-L\'{e}vy process. Reciprocally, if $\left( L_{t}\right) _{t\in 
\mathbb{R}}$ is a L\'{e}vy process on $\mathbb{R}$, the random measure
characterized by $L\left( \left( s,t\right] \right) :=L_{t}-L_{s}$ for $%
s\leq t$ is an homogeneous L\'{e}vy basis on $\mathbb{R}$. Therefore,
homogeneous L\'{e}vy bases and L\'{e}vy processes are in a bijection. More
generally, L\'{e}vy bases on $\mathbb{R}$ are in a bijection with \textit{%
natural additive processes}. See for instance \cite{satolevybaseiiprocesses}%
. For the multiparameter case, i.e. $\mathcal{S}=\mathbb{R}^{k}$ the same
result holds but for the so-called \emph{lamp processes}. See \cite%
{LevyItodecompositionrm} and \cite{AdMonScisWil(83)} for more detail about
these type of processes.

In analogy to subordinators on $\mathbb{R}^{+}$, let us consider the
following L\'{e}vy basis%
\begin{equation}
T\left( A\right) =\int_{A}\beta _{0}\left( s\right) c\left( ds\right)
+\int_{A}\int_{0}^{\infty }xN\left( dxds\right) \text{, \ \ }A\in \mathcal{B}%
_{b}\left( \mathbb{R}^{k}\right) ,  \label{eqn1.6}
\end{equation}%
with $\beta _{0}\geq 0$, $c$ being a continuous $\sigma $-finite measure on $%
\mathcal{B}\left( \mathbb{R}^{k}\right) $ and $N$ a Poisson random measure
with intensity $\nu \left( s,\mathrm{d}x\right) c\left( \mathrm{d}s\right) $%
. The measure $\nu$ satisfies that for $c$-almost all $s\in \mathbb{R}^{k}$, 
$\nu \left( s,\left( -\infty ,0\right) \right) =0$ and $\int_{0}^{\infty
}\left( 1\wedge x\right) \nu \left( s,\mathrm{d}x\right) <\infty $. Observe
that $T$ is non-negative and for almost all $\omega \in \Omega $, $T\left(
\omega ,\cdot \right) $ is in fact a true measure.

From \cite{AdMonScisWil(83)} and \cite{LevyItodecompositionrm}, $L$ is
almost surely a non-negative measure-valued field if and only if the
previously stated conditions hold. Being more precise, if $T$ is a L\'{e}vy
basis on $\mathbb{R}^{k}$ with characteristic quadruplet $\left( \beta
,b,\rho ,c\right) $ where $c$ is continuous, then for almost all $\omega \in
\Omega $, $T\left( \omega ,\cdot \right) $ is a true measure if and only if
the following conditions hold for $c$-almost all $s\in \mathbb{R}^{k}${}

\begin{enumerate}
\item $b\left( s\right) =0;$

\item $\nu \left( s,\left( -\infty ,0\right) \right) =0;$

\item $\int_{0}^{\infty }\left( 1\wedge x\right) \nu \left( s,\mathrm{d}%
x\right) <\infty ;$

\item $\beta \left( s\right) -\int_{0}^{\infty }x\rho \left( s,dx\right)
\geq 0.$
\end{enumerate}

In this case $T$ admits the representation given in (\ref{eqn1.6}) with $%
\beta _{0}\left( s\right) =\beta \left( s\right) -\int_{0}^{1}x\nu \left( s,%
\mathrm{d}x\right) $. Moreover, $L$ can be extended to the whole $\sigma $%
-algebra $\mathcal{B}\left( \mathbb{R}^{k}\right) $ by 
\begin{equation*}
T\left( A\right) :=\lim_{n\rightarrow \infty }T\left( A\cap S_{n}\right) ,%
\text{ \ \ }A\in \mathcal{B}\left( \mathbb{R}^{k}\right) \text{,}
\end{equation*}%
where the limit exists almost surely.

\subsection{Meta-times}

Meta-times where introduced in \cite{BNJP(12)} as a way to extend the
concept of a time change of L\'{e}vy processes (see for instance \cite%
{idprocessesmaejimasatoole}) to the spatiotemporal case. A \emph{meta-time} $%
\varphi :\mathcal{B}_{b}\left( \mathbb{R}^{k}\right) \rightarrow \mathcal{B}%
_{b}\left( \mathbb{R}^{k}\right) $ is a set function that maps disjoint sets
into disjoint sets and for any disjoint sequence $\left( A_{n}\right)
_{n\geq 1}\subset \mathcal{B}_{b}\left( \mathbb{R}^{k}\right) $ such that $%
\cup _{n\geq 1}A_{n}\in \mathcal{B}_{b}\left( \mathbb{R}^{k}\right) ,$ it
holds that $\varphi \left( \cup _{n\geq 1}A_{n}\right) =\cup _{n\geq
1}\varphi \left( A_{n}\right) .$ Note that if $\varphi $ is a meta-time,
then the random measure defined by%
\begin{equation}
L_{m}\left( A\right) :=L\left( \varphi \left( A\right) \right) ,\text{ \ \ }%
A\in \mathcal{B}_{b}\left( \mathbb{R}^{k}\right) \text{,}  \label{eqn1.7.0}
\end{equation}%
is a L\'{e}vy basis provided that $L$ is L\'{e}vy basis as well. If $L$ is
homogeneous, then $L_{m}$ satisfies 
\begin{equation}
\mathcal{L}\left[ L_{m}\left( A\right) \right] =\mu ^{m\left( A\right) }%
\text{, \ \ }A\in \mathcal{B}_{b}\left( \mathbb{R}^{k}\right) ,
\label{eqn1.7}
\end{equation}%
where $m$ stands for the measure induced by $\varphi $, this is $m:=Leb\circ
\varphi $. The most important result in \cite{idprocessesmaejimasatoole}
allows to get the reversal of the previous property: Let $m:\mathcal{B}%
\left( \mathbb{R}^{k}\right) \rightarrow \mathbb{R}^{+}$ be a measure which
is finite in $\mathcal{B}_{b}\left( \mathbb{R}^{k}\right) $ and such that $%
m\left( \left\{ x\in \mathbb{R}^{k}:x_{i}=0\text{ for some }i=1,\ldots
,k\right\} \right) =0$. Then, there exists a measurable function (not
necessarily unique) $\phi :\mathbb{R}^{k}\rightarrow \mathbb{R}^{k}$ such
that 
\begin{equation*}
m\left( A\right) =Leb\left( \phi ^{-1}\left( A\right) \right) ,\text{ \ \ }%
A\in \mathcal{B}\left( \mathbb{R}^{k}\right) .
\end{equation*}%
Putting $\varphi =\phi ^{-1}$, we see that in this case $L_{m}$ can be
defined as in (\ref{eqn1.7.0}), and (\ref{eqn1.7}) holds. Therefore, if $%
\left( T\left( A\right) \right) _{A\in \mathcal{B}\left( \mathbb{R}%
^{k}\right) }$ is a measure-valued stochastic field, for which almost surely%
\begin{equation}
T\left( \left\{ x\in \mathbb{R}^{k}:x_{i}=\text{ for some }i=1,\ldots
,k\right\} \right) =0,  \label{eqn1.8}
\end{equation}%
then for every $\omega \in \Omega _{0}$, with $\Omega _{0}$ being the set
where (\ref{eqn1.8}) holds and $T$ is a true measure, there exists a
measurable mapping $\phi _{\omega }:\mathbb{R}^{k}\rightarrow \mathbb{R}^{k}$
such that%
\begin{equation*}
T\left( \omega ,A\right) =Leb\left( \phi _{\omega }^{-1}\left( A\right)
\right) ,\text{ \ \ }A\in \mathcal{B}\left( \mathbb{R}^{k}\right) .
\end{equation*}%
Furthermore, if $T$ is independent of $L$, then we define $L_{T}$ to be the 
\emph{subordinated L\'{e}vy basis} by $T$ as the random measure defined as%
\begin{equation}
L_{T}\left( \omega ,A\right) :=L\left( \omega ,\varphi \left( \omega
,A\right) \right) ,\text{ \ \ }\omega \in \Omega _{0},A\in \mathcal{B}%
_{b}\left( \mathbb{R}^{k}\right) \text{,}  \label{eqn1.9.0}
\end{equation}%
and an arbitrary value outside of $\Omega _{0}$. In the previous equation, $%
\varphi \left( \omega ,A\right) =\phi _{\omega }^{-1}\left( A\right) $ is
the meta-time induced by $T\left( \omega ,\cdot \right) $. Note that by
equation (\ref{eqn1.7}), if $L$ is an homogeneous L\'{e}vy basis, then 
\begin{equation}
\mathcal{L}\left[ \left. L_{T}\left( A\right) \right\vert T\right] =\mu
^{T\left( A\right) }\text{, \ \ }A\in \mathcal{B}_{b}\left( \mathbb{R}%
^{k}\right) .  \label{eqn1.9}
\end{equation}%
However, in general, $L_{T}$ is not a L\'{e}vy basis: Put $L=Leb$ and 
\begin{equation*}
T\left( A\right) :=ULeb\left( A\right) ,\text{ \ \ }U\sim \text{Uniform}%
\left[ 0,1\right] \text{.}
\end{equation*}%
Then $T\left( \omega ,\cdot \right) $ is a true measure satisfying (\ref%
{eqn1.8}), but $L_{T}$ is not infinitely divisible. Indeed, observe that in
this case 
\begin{equation*}
L_{T}\left( A\right) =T\left( A\right) ,\text{ \ \ }A\in \mathcal{B}%
_{b}\left( \mathbb{R}^{k}\right) .
\end{equation*}%
which is not infinitely divisible. Nevertheless, if $T$ is a L\'{e}vy basis,
with minor changes in the proof of Theorem 5.1 in \cite{BNJP(12)}, we get
that:

\begin{theorem}[\protect\cite{BNJP(12)}]
\label{TheoremcqsubLB}Let $L$ be an homogeneous L\'{e}vy basis with
characteristic triplet $\left( \gamma ,b,\nu \right) $ and $T$ a
non-negative L\'{e}vy basis with characteristic quadruplet $\left( \gamma
,0,\rho ,c\right) $. Then $L_{T}$ as in (\ref{eqn1.9.0}) is a L\'{e}vy basis
with characteristic quadruplet $\left( \overline{\gamma },\overline{b},%
\overline{\nu },\overline{c}\right) $ given by

\begin{enumerate}
\item $\overline{c}=c;$

\item $\overline{\gamma }\left( s\right) =\gamma \beta _{0}\left( s\right)
+\int_{0}^{\infty }\int_{\left\vert x\right\vert \leq 1}x\mu ^{r}\left(
dx\right) \rho \left( s,dr\right) ;$

\item $\overline{b}\left( s\right) =b\beta _{0}\left( s\right) ;$

\item $\overline{\nu }\left( s,dx\right) =\beta _{0}\left( s\right) \nu
\left( dx\right) +\int_{0}^{\infty }\mu ^{r}\left( dx\right) \rho \left(
s,dr\right) ,$

where $\mu $ is the ID law associated to $\left( \gamma ,b,\nu \right) $ and 
$\beta _{0}\left( s\right) =\gamma \left( s\right) -\int_{0}^{\infty }x\rho
\left( s,dx\right) .$
\end{enumerate}
\end{theorem}

\begin{remark}
The previous result is the generalization of the corresponding law of a
subordinated L\'{e}vy process. See Theorem 30.1 in \cite{sato} for more
details. Hence, if $T$ is homogeneous, the triplet associated to $L_{T}$
corresponds to the law of a subordinated L\'{e}vy process on the real line.
\end{remark}

\section{A class of L\'{e}vy measures obtained by L\'{e}vy mixing}

In this part we study a certain  mapping obtained by mixing a family of
probability distributions through a L\'{e}vy measure.

Let $\mathfrak{\ M}_{L}\left( \mathbb{R}\right) $ denote the class of L\'{e}%
vy measures on $\mathbb{R}$, i.e.~a $\sigma $-finite measure $\rho$ on $%
\mathcal{B}\left( \mathbb{R}\right)$ belongs to $\mathfrak{M}_{L}\left( 
\mathbb{R}\right) $ if $\rho \left( \left\{ 0\right\} \right) =0$ and $\int_{%
\mathbb{R}}(1\wedge \left\vert x\right\vert ^{2})\rho \left( \mathrm{d}%
x\right) <\infty $. By $\mathfrak{M}_{L}^{1}\left( \mathbb{R}\right) \subset 
\mathfrak{M}_{L}\left( \mathbb{R}\right) $ and $\mathfrak{M}_{L}^{0}\left( 
\mathbb{R}\right) \subset \mathfrak{M}_{L}\left( \mathbb{R}\right) $ we mean
the subclasses of L\'{e}vy measures satisfying $\int_{\mathbb{R}}(1\wedge
\left\vert x\right\vert ) \rho \left( \mathrm{d}x\right) <\infty $ and $\rho
\left( \mathbb{R}\right) <\infty $, respectively. We define in a similar way 
$\mathfrak{M}_{L}\left( \mathbb{R}^{+}\right) $, $\mathfrak{M}_{L}^{1}\left( 
\mathbb{R}^{+}\right) $ and $\mathfrak{M}_{L}^{0}\left( \mathbb{R}%
^{+}\right) .$ Consider a measurable collection of probability measures $%
M=\left( \mu \left( s,\cdot \right) \right) _{s\geq 0}$. We introduce the
following mapping 
\begin{equation*}
\Phi _{M}\left( \rho \right) \left( A\right) :=\int_{0}^{\infty }\mu \left(
s,A\right) \rho \left( ds\right) ,\text{ \ \ }\rho \in \mathfrak{M}%
_{L}\left( \mathbb{R}^{+}\right) ,A\in \mathcal{B}\left( \mathbb{R}%
\backslash \left\{ 0\right\} \right) ,
\end{equation*}%
and $\Phi _{M}\left( \rho \right) \left( \left\{ 0\right\} \right) =0.$ If $%
\Phi _{M}\left( \rho \right) \in \mathfrak{M}_{L}\left( \mathbb{R}\right) $,
then it belongs to the subclass of L\'{e}vy measures appearing by \textit{L%
\'{e}vy mixing}. See \cite{BNPabreuThorbS2013} for more examples of L\'{e}vy
mixing.

Note that when $\mu \left( s,\cdot \right) $ is the law of a L\'{e}vy
process at time $s,$ $\Phi _{M}$ corresponds to the L\'{e}vy measure of a
subordinated L\'{e}vy process by a subordinator having no drift and L\'{e}vy
measure $\rho $. In the same context, $\Phi _{M}$ also describes the mixed
probability distribution obtained by subordinating a L\'{e}vy process trough
a random time. Being more precise, if $\mu \left( s,\cdot \right) $ is the
law of a L\'{e}vy process $L$ at time $s$ and $\rho $ is a probability
measure with support on $\mathbb{R}^{+}$, then $\Phi _{M}\left( \rho \right) 
$ is the distribution of the random variable $L_{T^{\rho }}$, for a random
time $T^{\rho }\sim \rho $ and independent of $L$. Indeed, by independence 
\begin{eqnarray}
\mathbb{P}\left( L_{T^{\rho }}\in A\right) &=&\int_{0}^{\infty }\mathbb{P}%
\left( L_{s}\in A\right) \rho \left( ds\right)  \label{eqn2.0} \\
&=&\Phi _{M}\left( \rho \right) \left( A\right) \text{, \ \ }A\in \mathcal{B}%
\left( \mathbb{R}\backslash \left\{ 0\right\} \right) \text{.}  \notag
\end{eqnarray}

However, such a mapping is not only limited to a family of infinitely
divisible distributions. For instance, if we choose 
\begin{equation*}
\mu \left( s,dx\right) =\frac{1}{2\pi }\left( s-x^{2}\right) ^{-1/2}\mathbf{1%
}_{\left( 0,s^{1/2}\right) }\left( x\right) dx,\text{ \ \ }s\geq 0,
\end{equation*}%
the mapping $\Phi _{M}$ is well defined on $\mathfrak{M}_{L}^{1}\left( 
\mathbb{R}^{+}\right) $ and it describes all the L\'{e}vy measures of the
so-called class $A\left( \mathbb{R}\right) $. Moreover, $\Phi _{M}$ is
one-to-one and it is not an \textit{upsilon transformation}. That is, there
is no $\sigma $-finite measure $\eta $ such that 
\begin{equation*}
\Phi _{M}\left( \rho \right) \left( A\right) =\int_{0}^{\infty }\eta \left(
s^{-1}A\right) \rho \left( ds\right) ,\text{ \ \ }\rho \in \mathfrak{M}%
_{L}\left( \mathbb{R}^{+}\right) ,A\in \mathcal{B}\left( \mathbb{R}%
\backslash \left\{ 0\right\} \right) ,
\end{equation*}%
see \cite{ArizBNPAbreu2010} and \cite{MaejPAbreuSato2012} for more details
on these aspects as well as \cite{MaejSatoPAbreu2013} for generalizations.
For a comprehensive introduction to upsilon transformations we refer to \cite%
{MR2421179}. It is remarkable that in this case $\mu \left( s,\cdot \right) $
is not infinitely divisible in the classical sense, but according to \cite%
{FrUNaM2005}, $\mu \left( s,\cdot \right) $ is infinitely divisible with
respect to the \textit{monotone convolution}. In fact, such a distribution
plays the role of the Gaussian distribution under this operation. In
contrast, if 
\begin{equation*}
\mu \left( s,dx\right) =\frac{1}{\sqrt{2\pi s}}e^{-\frac{x^{2}}{2s}}dx,\text{
\ \ }s\geq 0,
\end{equation*}%
the induced mapping $\Phi _{M},$ as we show later, is actually equivalent to
an upsilon transformation, it is one-to-one and corresponds to the L\'{e}vy
measure of a subordinated Brownian motion.

All the examples above show the importance of $\Phi _{M}$. Hence, in this
section we study some properties of this transformation including its
injectivity.

In this paper, we are interested in the case when $\mu \left( s,dx\right)
=\mu ^{s}\left( dx\right) $ for $\mu \in ID\left( \mathbb{R}\right) .$ In
this context, we will write $\Phi _{\mu }$ instead of $\Phi _{M}$. We focus
on the following three cases:

\begin{enumerate}
\item The restriction of $\Phi _{\mu}$ to $\mathfrak{M}_{L}^{0}\left( 
\mathbb{R}^{+}\right) ;$

\item $\mu $ is symmetric;

\item $\mu $ having support on $\left[ 0,\infty \right) $.
\end{enumerate}

Surprisingly, the one-to-one property of $\Phi _{\mu }$ in cases $2$ and $3$%
, as we show later, is equivalent to the case $1$. We start by describing
the domain and the range of $\Phi _{\mu }$.

\subsection{The domain and the range of $\Phi _{\protect\mu }$}

For a given $\mu \in ID\left( \mathbb{R}\right) ,$ consider the L\'{e}%
vy-domain $\Phi _{\mu }$, denoted by 
\begin{equation*}
\mathcal{D}_{L}\left( \Phi _{\mu }\right) :=\left\{ \rho \in \mathfrak{M}%
_{L}\left( \mathbb{R}^{+}\right) :\Phi _{\mu }\left( \rho \right) \in 
\mathfrak{M}_{L}\left( \mathbb{R}\right) \right\} ,
\end{equation*}%
and the range of $\Phi _{\mu }$, denoted by 
\begin{equation*}
\mathcal{R}_{L}\left( \Phi _{\mu }\right) =\left\{ \Phi _{\mu }\left( \rho
\right) \in \mathfrak{M}_{L}\left( \mathbb{R}\right) :\rho \in \mathcal{D}%
_{L}\left( \Phi _{\mu }\right) \right\}.
\end{equation*}
In this subsection we describe both $\mathcal{D}_{L}\left( \Phi _{\mu
}\right) $ and $\mathcal{R}_{L}\left( \Phi _{\mu }\right) .$ In order to do
that, we prepare a lemma.

\begin{lemma}
Let $\mu \in ID\left( \mathbb{R}\right) $ be a non-degenerate probability
measure. We have that the following asymptotic holds%
\begin{equation*}
\int_{\mathbb{R}}\left( 1\wedge \left\vert x\right\vert ^{2}\right) \mu
^{s}\left( dx\right) \sim cs\text{, \ \ as }s\downarrow 0,
\end{equation*}%
for some $c>0.$
\end{lemma}

\begin{proof}
Let $\mu $ be a non-degenerate probability measure on $ID\left( \mathbb{R}%
\right) $ with characteristic triplet $\left( \gamma ,b,\nu \right) $. It is
a well known fact that if $s_{n}\downarrow 0,$ then the sequence of
probability measures having characteristic triplet $\left( 0,0,\frac{1}{s_{n}%
}\mu ^{s_{n}}\right) _{0}$ converges weakly to $\mu $. From the proof of
Theorem 8.7 in \cite{sato}, we get that in this case, the sequence of finite
measures $\rho _{n}\left( dx\right) :=\frac{1}{s_{n}}\left( 1\wedge
\left\vert x\right\vert ^{2}\right) \mu ^{s_{n}}\left( dx\right) $ converges
weakly to the finite measure%
\begin{equation*}
\rho \left( dx\right) :=b\delta _{0}\left( dx\right) +\left( 1\wedge
\left\vert x\right\vert ^{2}\right) \nu \left( dx\right) .
\end{equation*}%
In particular, $\frac{1}{s_{n}}\int_{\mathbb{R}}\left( 1\wedge \left\vert
x\right\vert ^{2}\right) \mu ^{s_{n}}\left( dx\right) \rightarrow b+\int_{%
\mathbb{R}\backslash \left\{ 0\right\} }\left( 1\wedge \left\vert
x\right\vert ^{2}\right) \nu \left( dx\right) >0$ as $n\rightarrow \infty ,$
as required.
\end{proof}

Now we proceed to describe the domain of $\Phi _{\mu }$.

\begin{proposition}
\label{popdomrange}Let $\mu \in ID\left( \mathbb{R}\right) $. Then

\begin{enumerate}
\item $\mathcal{D}_{L}\left( \Phi _{\mu }\right) =\mathfrak{M}_{L}^{1}\left( 
\mathbb{R}^{+}\right) $ provided that $\mu $ is not a Dirac delta measure.

\item In general $\mathcal{R}_{L}\left( \Phi _{\mu }\right) \subsetneq 
\mathfrak{M}_{L}\left( \mathbb{R}\right) $.
\end{enumerate}
\end{proposition}

\begin{proof}
$1.$ Assume that $\mu $ is not a Dirac delta measure. Let $\overline{\nu }%
=\Phi _{\mu }\left( \rho \right) $ for some $\rho \in \mathfrak{M}_{L}\left( 
\mathbb{R}^{+}\right) .$ From the previous lemma, there is $s_{0}>0$ and
constants $C_{1},C_{2}>0$ such that 
\begin{equation*}
C_{1}\int_{0}^{s_{0}}s\rho \left( ds\right) \leq \int_{0}^{s_{0}}\int_{%
\mathbb{R}}\left( 1\wedge \left\vert x\right\vert ^{2}\right) \mu ^{s}\left(
dx\right) \rho \left( ds\right) \leq C_{2}\int_{0}^{s_{0}}s\rho \left(
ds\right) .
\end{equation*}%
Therefore, $\int_{\mathbb{R}}\left( 1\wedge \left\vert x\right\vert
^{2}\right) \overline{\nu }\left( dx\right) <\infty $ if and only if $%
\int_{0}^{\infty }\left( 1\wedge s\right) \rho \left( ds\right) <\infty ,$
as required.

$2.$ This follows by observing that if $\mu $ has a density, then $\Phi
_{\mu }\left( \rho \right) $ has a density as well for any $\rho \in 
\mathcal{D}_{L}\left( \Phi _{\mu }\right) $.
\end{proof}

\begin{remark}
Observe that when $\mu \left( dx\right) =\delta _{\left\{ \gamma \right\}
}\left( dx\right) $ for some $\gamma \in \mathbb{R}$, then 
\begin{equation*}
\Phi _{\mu }\left( \rho \right) \left( dx\right) =\left\{ 
\begin{array}{cc}
\rho \left( \gamma ^{-1}dx\right) & \text{if }\gamma \neq 0; \\ 
0 & \text{otherwise.}%
\end{array}%
\right.
\end{equation*}%
This means that the properties of $\Phi _{\mu }$ are the properties of $%
\mathfrak{M}_{L}\left( \mathbb{R}^{+}\right) .$ Hence, for the rest of the
paper we will only consider non-degenerated infinitely divisible
distributions.
\end{remark}

\subsection{The mapping $\Phi _{\protect\mu }$ restricted to $\mathfrak{M}%
_{L}^{0}\left( \mathbb{R}^{+}\right) $}

In this part we will show that for a given $\mu \in ID\left( \mathbb{R}%
\right) $, the mapping $\Phi _{\mu }$ restricted to $\mathfrak{M}%
_{L}^{0}\left( \mathbb{R}^{+}\right) $ is one-to-one.

Before we present the main result let us show some properties of the mapping 
$\Phi _{\mu }$ restricted to $\mathfrak{M}_{L}^{0}\left( \mathbb{R}%
^{+}\right) $, which we denote by $\Phi _{\mu }^{0}$.

\begin{proposition}
\label{poplaplace}Let $\mu \in ID\left( \mathbb{R}\right) $ with
characteristic exponent $\phi $. Then $\Phi _{\mu }^{0}$ has the following
properties:

\begin{enumerate}
\item We have that $\Phi _{\mu }^{0}\left( \mathfrak{M}_{L}^{0}\left( 
\mathbb{R}^{+}\right) \right) \subsetneq \mathfrak{M}_{L}^{0}\left( \mathbb{R%
}^{+}\right) $.

\item If $\nu =\Phi _{\mu }^{0}\left( \rho \right) $ for some $\rho \in 
\mathfrak{M}_{L}^{0}\left( \mathbb{R}^{+}\right) $, then any continuous and
bounded complex-valued function is $\nu $-integrable and%
\begin{equation}
\int_{\mathbb{R}}e^{i\theta x}\nu \left( dx\right) =\int_{0}^{\infty
}e^{s\phi \left( \theta \right) }\rho \left( ds\right) \text{, \ \ }\theta
\in \mathbb{R}\text{.}  \label{eqn2.7}
\end{equation}
\end{enumerate}
\end{proposition}

\begin{proof}
$1$. Since for any $\rho \in \mathfrak{M}_{L}^{0}\left( \mathbb{R}%
^{+}\right) ,$ $\Phi _{\mu }^{0}\left( \rho \right) \left( \mathbb{R}\right)
=\rho \left( \mathbb{R}^{+}\right) <\infty ,$ we get that $\Phi _{\mu
}^{0}\left( \mathfrak{M}_{L}^{0}\left( \mathbb{R}^{+}\right) \right) \subset 
\mathfrak{M}_{L}^{0}\left( \mathbb{R}^{+}\right) $. The strictly inclusion
follows from Proposition \ref{popdomrange}.

$2.$ Let $\nu =\Phi _{\mu }^{0}\left( \rho \right) $ for some $\rho \in 
\mathfrak{M}_{L}^{0}\left( \mathbb{R}^{+}\right) .$\ From the previous point 
$\Phi _{\mu }^{0}\left( \rho \right) \in \mathfrak{M}_{L}^{0}\left( \mathbb{R%
}^{+}\right) ,$ thus for any real-valued continuous and bounded function $f$
we have that 
\begin{equation*}
\int_{\mathbb{R}}\left\vert f\left( x\right) \right\vert \nu \left(
dx\right) =\int_{0}^{\infty }\int_{\mathbb{R}}\left\vert f\left( x\right)
\right\vert \mu ^{s}\left( dx\right) \rho \left( ds\right) <\infty .
\end{equation*}%
If $f$ is complex-valued, continuous and bounded, its real part and its
imaginary part are continuous and bounded as well, so integrable. Finally,
since $\mu ^{s}$ is the probability measure with characteristic function $%
e^{s\phi \left( \theta \right) }$, (\ref{eqn2.7}) follows easily.
\end{proof}

\begin{remark}
Observe that equation (\ref{eqn2.7}) only holds when $\rho \in \mathfrak{M}%
_{L}^{0}\left( \mathbb{R}^{+}\right) $. Indeed, put $\rho \left( dx\right)
=x^{-\alpha -1}dx$ for $x>0$ and $0<\alpha <1$. Then $\rho \in \mathfrak{M}%
_{L}^{1}\left( \mathbb{R}^{+}\right) $ but $\rho \notin \mathfrak{M}%
_{L}^{0}\left( \mathbb{R}^{+}\right) $ and for any $z\in \mathbb{C}$ with $%
\text {Re}z\leq 0$, $\int_{0}^{\infty }e^{zs}\rho \left( ds\right) $ does
not exist.
\end{remark}

Next, we present the one-to-one property of $\Phi _{\mu }^{0}.$ Recall that $%
x_{0}\in A\subset \mathbb{C}$ is a condensation point if for every $%
\varepsilon >0,$ $B_{\varepsilon }\left( x_{0}\right) \cap A\backslash
\left\{ x_{0}\right\} \neq \emptyset $, where $B_{\varepsilon }\left(
x_{0}\right) $ is the open ball in $\mathbb{C}$ with radius $\varepsilon $
and center $x_{0}$.

\begin{theorem}
\label{Thmcompoundpoisson}Let $\mu \in ID\left( \mathbb{R}\right) $ with
characteristic triplet $\left( \gamma ,b,\nu \right) $. Then, the mapping $%
\Phi _{\mu }^{0}$ is one-to-one on its own domain.
\end{theorem}

\begin{proof}
Let $\rho ,\widetilde{\rho }\in \mathfrak{M}_{L}^{0}\left( \mathbb{R}%
^{+}\right) $ such that $\Phi _{\mu }^{0}\left( \rho \right) =\Phi _{\mu
}^{0}\left( \widetilde{\rho }\right) $. From the proof of Proposition \ref%
{poplaplace}, $\rho $ and $\widetilde{\rho }$ are finite measures with the
same mass, so without loss of generality we may and do assume that they are
probability measures. By equation (\ref{eqn2.7})%
\begin{equation*}
\int_{0}^{\infty }e^{sz}\rho \left( ds\right) =\int_{0}^{\infty }e^{sz}%
\widetilde{\rho }\left( ds\right) \text{, \ \ }\forall \text{ }z\in \mathcal{%
A}\text{,}
\end{equation*}%
where%
\begin{equation}
\mathcal{A}:=\left\{ \phi \left( \theta \right) :\theta \in \mathbb{R}%
\right\} .  \label{eqn2.8}
\end{equation}%
We claim that there is $\theta _{0}\in \mathbb{R}$ for which $\text{Re}\phi
\left( \theta _{0}\right) <0$. If this were true, due to the continuity of $%
\phi $, $\phi \left( \theta _{0}\right) $ would be a condensation point of $%
\mathcal{A}$ having $\text{Re}\phi \left( \theta _{0}\right) <0$ and since
the set $\left( -\infty ,0\right) \times i\mathbb{R}$ is contained in the
interior of the domain of the Laplace transforms of $\rho $ and $\widetilde{%
\rho }$, this would imply that $\rho $ and $\widetilde{\rho }$ coincides
(see Chapter 4 in \cite{hoffmanprob}), finishing thus the proof. Hence, we
only need to show that such a $\theta _{0}$ exists. Suppose the opposite,
this is $\text{Re}\phi \geq 0$. Since for any $\theta \in \mathbb{R}$%
\begin{equation*}
\text{Re}\phi \left( \theta \right) =-\frac{1}{2}b\theta ^{2}-\int_{\mathbb{R%
}\backslash \left\{ 0\right\} }\left[ 1-\cos \left( \theta x\right) \right]
\nu \left( dx\right) \leq 0,
\end{equation*}%
we conclude that $\text{Re}\phi =0$ which is equivalent to saying that $b=0$
and 
\begin{equation*}
\int_{\mathbb{R}\backslash \left\{ 0\right\} }\left[ 1-\cos \left( \theta
x\right) \right] \nu \left( dx\right) =0\text{, \ \ }\forall \text{\ }\theta
\in \mathbb{R}\text{.}
\end{equation*}%
Since $1-\cos \left( \theta x\right) \geq 0,$ from the previous equation we
deduce that $\cos \left( \theta \cdot \right) \equiv 1$ $\nu $-a.e. and for
all $\theta \in \mathbb{R}$, which is impossible. Thus, such a $\theta _{0}$
must exists.
\end{proof}

\begin{corollary}[Probability measures]
\label{injectprob}Let $\mathcal{P}\left( \mathbb{R}^{+}\right) $ be the
space of probability measures on $\mathbb{R}^{+}$. Then $\Phi _{\mu }$
restricted to $\mathcal{P}\left( \mathbb{R}^{+}\right) $ is one-to-one.
\end{corollary}

\subsection{The symmetric case}

The situation of $\mu $ being symmetric is of great interest, not only for
its tractability but also because in this case $\Phi _{\mu }$ is injective
in its whole domain.

The following theorem is the main result of this subsection.

\begin{theorem}
\label{Thmalphastlm}Suppose that $\mu \in ID\left( \mathbb{R}\right) $ is
symmetric. Then $\Phi _{\mu }$ is one-to-one on its own domain.
\end{theorem}

This result is basically a consequence of Corollary \ref{injectprob} and the
following proposition:

\begin{proposition}
\label{keyprop}Assume that $\mu \in ID\left( \mathbb{R}\right) $ is
symmetric and let $\rho ,\widetilde{\rho }\in \mathfrak{M}_{L}^{1}\left( 
\mathbb{R}^{+}\right) $. Then $\Phi _{\mu }\left( \rho \right) =\Phi _{\mu
}\left( \widetilde{\rho }\right) $ if and only if there exist (not
necessarily unique) $\eta _{\rho },\eta _{\widetilde{\rho }}\in ID\left( 
\mathbb{R}^{+}\right) $ having L\'{e}vy measures $\rho $ and $\widetilde{%
\rho },$ respectively, such that $\Phi _{\mu }\left( \eta _{\rho }\right)
=\Phi _{\mu }\left( \eta _{\widetilde{\rho }}\right) $.
\end{proposition}

\begin{proof}
Let us start by assuming that there are two infinitely divisible
distributions $\eta _{\rho },\eta _{\widetilde{\rho }}\in ID\left( \mathbb{R}%
^{+}\right) $ with L\'{e}vy measures $\rho $ and $\widetilde{\rho }$, such
that $\Phi _{\mu }\left( \eta _{\rho }\right) =\Phi _{\mu }\left( \eta _{%
\widetilde{\rho }}\right) $. From Corollary \ref{injectprob}, it follows
immediately that $\eta _{\rho }=\eta _{\widetilde{\rho }}$. Thus, by
uniqueness of the characteristic triplet, it follows that $\rho =\widetilde{%
\rho }$.

Conversely, suppose that $\Phi _{\mu }\left( \rho \right) =\Phi _{\mu
}\left( \widetilde{\rho }\right) $. First note that in general $\Phi _{\mu
}\left( \rho \right) \equiv 0$ if and only if $\rho \equiv 0$, so without
loss of generality, we may and do assume that $\rho ,\widetilde{\rho }\neq 0$%
. Let, $T^{\rho }$ and $T^{\widetilde{\rho }}$ be two purely discontinuous
subordinators with L\'{e}vy measures $\rho $ and $\widetilde{\rho }$,
respectively. Put $\eta _{\rho }$ and $\eta _{\widetilde{\rho }}$ to be the
distributions related to $T^{\rho }$ and $T^{\widetilde{\rho }}$ at time 1$.$
Let $\overline{\mu }=\mathcal{L}\left( L_{T_{1}^{\rho }}\right) $ and $%
\overline{\mu }^{\prime }=\mathcal{L}\left( L_{T_{1}^{\widetilde{\rho }%
}}\right) $, where $L$ is the L\'{e}vy process associated to $\mu $. Since $%
\overline{\mu }=\Phi _{\mu }\left( \eta _{\rho }\right) $ and $\overline{\mu 
}^{\prime }=\Phi _{\mu }\left( \eta _{\widetilde{\rho }}\right) $, then $%
\Phi _{\mu }\left( \eta _{\rho }\right) =\Phi _{\mu }\left( \eta _{%
\widetilde{\rho }}\right) $ if and only if the characteristic triplets of $%
\overline{\mu }$ and $\overline{\mu }^{\prime }$ coincide. From Theorem \ref%
{TheoremcqsubLB} we get that $\overline{\mu }$ and $\overline{\mu }^{\prime
} $ have characteristic triplets $\left( \gamma ,0,\nu \right) $ and $\left(
\gamma ^{\prime },0,\nu ^{\prime }\right) $ given by $\nu =\Phi _{\mu
}\left( \rho \right) ,$ $\nu ^{\prime }=\Phi _{\mu }\left( \widetilde{\rho }%
\right) $ and%
\begin{equation*}
\gamma =\int_{0}^{\infty }\int_{\left\vert x\right\vert \leq 1}x\mu
^{s}\left( dx\right) \rho \left( ds\right) \text{; \ \ }\gamma ^{\prime
}=\int_{0}^{\infty }\int_{\left\vert x\right\vert \leq 1}x\mu ^{s}\left(
dx\right) \widetilde{\rho }\left( ds\right) .
\end{equation*}%
By hypothesis $\nu =\widetilde{\nu }$. Moreover, since $\mu $ is symmetric $%
\int_{\left\vert x\right\vert \leq 1}x\mu ^{s}\left( dx\right) =0$ for any $%
s>0,$ it follows that $\gamma =\gamma ^{\prime }=0$, which means that $%
\overline{\mu }=\overline{\mu }^{\prime }$, as required. To finish the
proof, note that $\eta _{\rho }$ and $\eta _{\widetilde{\rho }}$ are not
unique. To see this, it is enough to consider $T^{\rho }$ and $T^{\widetilde{%
\rho }}$ to be two subordinators with the same drift and L\'{e}vy measures $%
\rho $ and $\widetilde{\rho }$. In this case, by following the same
reasoning as above, we get that $\Phi _{\mu }\left( \eta _{\rho }\right)
=\Phi _{\mu }\left( \eta _{\widetilde{\rho }}\right) $, where, as before, $%
\eta _{\rho }$ and $\eta _{\widetilde{\rho }}$ are the laws of $T^{\rho }$
and $T^{\widetilde{\rho }}$, respectively.
\end{proof}

Now we present a proof of Theorem \ref{Thmalphastlm}:\smallskip

\begin{proof}[Proof of Theorem \protect\ref{Thmalphastlm}]
Assume that $\mu $ is symmetric. Let $\rho ,\widetilde{\rho }\in \mathfrak{M}%
_{L}^{1}\left( \mathbb{R}^{+}\right) $ such that $\Phi _{\mu }\left( \rho
\right) =\Phi _{\mu }\left( \widetilde{\rho }\right) $. From the previous
proposition there exist $\eta _{\rho },\eta _{\widetilde{\rho }}\in ID\left( 
\mathbb{R}^{+}\right) $ having L\'{e}vy measures $\rho $ and $\widetilde{%
\rho },$ respectively, such that $\Phi _{\mu }\left( \eta _{\rho }\right)
=\Phi _{\mu }\left( \eta _{\widetilde{\rho }}\right) $. An application of
Corollary \ref{injectprob} results in $\eta _{\rho }=\eta _{\widetilde{\rho }%
}$, which obviously implies that $\rho =\widetilde{\rho }$.
\end{proof}

\begin{remark}
\label{remarkinffint}Due to Proposition \ref{keyprop} and Theorem \ref%
{Thmalphastlm} we have that the injectivity of $\Phi _{\mu }$\ is equivalent
to the one of $\Phi _{\mu }^{0}$ in the symmetric case.
\end{remark}

Recall that the Upsilon transformation of $\eta $ via $\rho $ is defined as
the $\sigma $-finite measure given by $\Upsilon _{\rho }\left( \eta \right)
\left( \left\{ 0\right\} \right) =0$ and%
\begin{equation*}
\Upsilon _{\rho }\left( \eta \right) \left( A\right) =\int_{0}^{\infty }\eta
\left( s^{-1}A\right) \rho \left( ds\right) ,\text{ \ \ }A\in \mathcal{B}%
\left( \mathbb{R}\backslash \left\{ 0\right\} \right) .
\end{equation*}

\begin{example}
\label{examplestable}Suppose that $\mu \in ID\left( \mathbb{R}\right) $ is a
symmetric strictly $\alpha $-stable distribution with $0<\alpha \leq 2$. By
the strict stability, we have that for all $s>0$%
\begin{equation*}
\mu ^{s}\left( dx\right) =\mu \left( s^{-1/\alpha }dx\right) .
\end{equation*}%
Thus, if $\rho \in \mathfrak{M}_{L}^{1}\left( \mathbb{R}^{+}\right) $ 
\begin{eqnarray*}
\Phi _{\mu }\left( \rho \right) \left( dx\right) &=&\int_{0}^{\infty }\mu
\left( s^{-1/\alpha }dx\right) \rho \left( ds\right) \\
&=&\int_{0}^{\infty }\mu \left( s^{-1}dx\right) \rho _{\alpha }\left(
ds\right) ,
\end{eqnarray*}%
where $\rho _{\alpha }$ is the measure induced by $\rho $ and the mapping $%
s\longmapsto s^{1/\alpha }$. Therefore $\Phi _{\mu }\left( \rho \right) $
equals $\Upsilon _{\rho _{\alpha }}\left( \mu \right) $, the upsilon
transformation of $\mu $ via $\rho _{\alpha }$. In particular, when $\mu $
is concentrated in $\left( 0,\infty \right) $%
\begin{equation}
\Phi _{\mu }\left( \rho \right) =\Upsilon _{\mu }\left( \rho _{\alpha
}\right) =\mu \circledast \rho _{\alpha },\text{ \ \ }\rho \in \mathfrak{M}%
_{L}^{1}\left( \mathbb{R}^{+}\right) \text{.}  \label{eqn2.9}
\end{equation}%
where $\circledast $ is the multiplicative convolution of measures, this is 
\begin{equation*}
\nu \circledast \eta \left( A\right) =\int_{0}^{\infty }\int_{0}^{\infty
}1_{A}\left( xs\right) \eta \left( dx\right) \nu \left( ds\right) \text{, \
\ }A\in \mathcal{B}\left( \mathbb{R}^{+}\right) \text{,}
\end{equation*}%
for two $\sigma $-finite measures on $\mathcal{B}\left( \mathbb{R}%
^{+}\right) $. In this context, the injectivity of $\Phi _{\mu }$ is
equivalent to the injectivity of $\Upsilon _{\mu }$ and the \textit{%
cancellation property} of $\mu $ (see for instance \cite{16479436} and \cite%
{onlssgamma}). In the next subsection we study these problems.
\end{example}

\subsection{The case of $\protect\mu $ supported on $\left( 0,\infty \right) 
$}

In the symmetric case, the key point for the injectivity of $\Phi _{\mu }$
is Proposition \ref{keyprop}. In this subsection we will show that such
result also holds in the case when $\mu $ is concentrated on $\left(
0,\infty \right) $.

\begin{proposition}
\label{keyprop2}Suppose that $\mu \in ID\left( \mathbb{R}\right) $ has
support on $\left( 0,\infty \right) $. Then the statement of Proposition \ref%
{keyprop} holds.
\end{proposition}

\begin{proof}
The argument in the proof of Proposition \ref{keyprop} remains the same in
this case, except that $\int_{\left\vert x\right\vert \leq 1}x\mu ^{s}\left(
dx\right) =0$ for any $s>0$, which of course is not true when $\mu $ has
support on $\left( 0,\infty \right) $. Therefore, under the notation of the
proof of Proposition \ref{keyprop}, we only need to show that $\gamma
=\gamma ^{\prime }$. Since $\mu $ has support on $\left( 0,\infty \right) $,
it corresponds to the law of a subordinator. Consequently, $\overline{\mu }$
and $\overline{\mu }^{\prime }$ are also concentrated on $\left( 0,\infty
\right) $ and for any non-negative measurable function $f$ vanishing on a
neighborhood of zero we have that 
\begin{equation*}
\int_{\mathbb{R}}f\left( x\right) \nu \left( dx\right) =\int_{0}^{\infty
}\int_{0}^{\infty }f\left( x\right) \mu ^{s}\left( dx\right) \rho \left(
ds\right) \text{,}
\end{equation*}%
and 
\begin{equation*}
\int_{\mathbb{R}}f\left( x\right) \nu ^{\prime }\left( dx\right)
=\int_{0}^{\infty }\int_{0}^{\infty }f\left( x\right) \mu ^{s}\left(
dx\right) \widetilde{\rho }\left( ds\right) .
\end{equation*}%
Thus, by the Monotone Convergence Theorem%
\begin{equation*}
\gamma =\lim_{\varepsilon \downarrow 0}\int_{0}^{\infty }\int_{\varepsilon
}^{1}x\mu ^{s}\left( dx\right) \rho \left( ds\right) =\lim_{\varepsilon
\downarrow 0}\int_{\varepsilon }^{1}x\nu \left( dx\right) ,
\end{equation*}%
and 
\begin{equation*}
\gamma ^{\prime }=\lim_{\varepsilon \downarrow 0}\int_{0}^{\infty
}\int_{\varepsilon }^{1}x\mu ^{s}\left( dx\right) \widetilde{\rho }\left(
ds\right) =\lim_{\varepsilon \downarrow 0}\int_{\varepsilon }^{1}x\nu
^{\prime }\left( dx\right) ,
\end{equation*}%
but by hypothesis $\nu =\nu ^{\prime },$ consequently $\gamma =\gamma
^{\prime }$, as required.
\end{proof}

\bigskip

By the previous proposition and Remark \ref{remarkinffint}, the
invertibility of $\Phi _{\mu }$ is equivalent to the injectivity of $\Phi
_{\mu }^{0}$, or in other words:

\begin{theorem}
\label{Thmgamma}Let $\mu \in ID\left( \mathbb{R}\right) $ having support on $%
\left( 0,\infty \right) .$ Then, $\Phi _{\mu }$ is one-to-one on its own
domain and $\Phi _{\mu }\left( \mathfrak{M}_{L}^{1}\left( \mathbb{R}%
^{+}\right) \right) \subsetneq \mathfrak{M}_{L}^{1}\left( \mathbb{R}%
^{+}\right) $.
\end{theorem}

\begin{example}[Gamma process]
Suppose that%
\begin{equation}
\mu \left( dx\right) =\lambda e^{-\lambda x}\mathbf{1}_{\left( 0,\infty
\right) }\left( x\right) dx\text{, \ \ }\lambda >0\text{.}  \label{eqn2.4}
\end{equation}%
From the previous theorem, the induced mapping $\Phi _{\mu }$ is one-to-one.
However, a direct proof can be given: We know that in this case, for any $%
s>0 $ 
\begin{equation*}
\mu ^{s}\left( dx\right) =\frac{\left( \lambda x\right) ^{s}}{x\Gamma \left(
s\right) }e^{-\lambda x}\mathbf{1}_{\left( 0,\infty \right) }\left( x\right)
dx.
\end{equation*}%
Thus, $\Phi _{\mu }\left( \rho \right) $ is absolutely continuous with
density given by $f_{\rho }\left( 0\right) =0$ and for $x>0$%
\begin{equation}
f_{\rho }\left( x\right) =\frac{e^{-\lambda x}}{x}\int_{0}^{\infty }\frac{%
\left( \lambda x\right) ^{s}}{\Gamma \left( s\right) }\rho \left( ds\right) .
\label{eqn2.5}
\end{equation}%
We claim that the transformation $\rho \longmapsto f_{\rho }$ is uniquely
determined by $\rho $. Let $\widetilde{\rho }\in \mathfrak{M}_{L}^{1}\left( 
\mathbb{R}^{+}\right) ,$ such that $f_{\rho }=f_{\widetilde{\rho }}.$ Then, $%
f_{\rho }\left( 0\right) =f_{\widetilde{\rho }}\left( 0\right) $\ and from (%
\ref{eqn2.5}), for all $x>0$%
\begin{equation}
\int_{0}^{\infty }\frac{\left( \lambda x\right) ^{s}}{\Gamma \left( s\right) 
}\rho \left( ds\right) =\int_{0}^{\infty }\frac{\left( \lambda x\right) ^{s}%
}{\Gamma \left( s\right) }\widetilde{\rho }\left( ds\right) .  \label{eqn2.6}
\end{equation}%
Consider the measures 
\begin{equation*}
\rho _{\Gamma }\left( ds\right) =\frac{1}{\Gamma \left( s\right) }\rho
\left( ds\right) \text{; \ }\widetilde{\rho }_{\Gamma }\left( ds\right) =%
\frac{1}{\Gamma \left( s\right) }\widetilde{\rho }\left( ds\right) .
\end{equation*}%
Since the gamma function behaves as $s^{-1}$ near zero, we have that $\rho
_{\Gamma }$ and $\widetilde{\rho }_{\Gamma }$ are finite measures.
Furthermore, from (\ref{eqn2.6}), $\rho _{\Gamma }$ and $\widetilde{\rho }%
_{\Gamma }$ have the same mass, so without loss of generality we may and do
assume that $\rho _{\Gamma }$ and $\widetilde{\rho }_{\Gamma }$ are
probability measures satisfying 
\begin{equation*}
\int_{0}^{\infty }x^{s}\rho _{\Gamma }\left( ds\right) =\int_{0}^{\infty
}x^{s}\widetilde{\rho }_{\Gamma }\left( ds\right) ,\text{ \ \ for all }x>0.
\end{equation*}%
Thus, by the uniqueness of the generating function (see Chapter 4 in \cite%
{hoffmanprob}), we conclude that $\rho _{\Gamma }=\widetilde{\rho }_{\Gamma
} $ or $\rho =\widetilde{\rho },$ because the gamma function is continuous
and strictly positive in $\left( 0,\infty \right) $.
\end{example}

\begin{example}[$\protect\alpha $-stable distribution]
Suppose that $\mu \in ID\left( \mathbb{R}^{+}\right) $ is a strictly $\alpha 
$-stable distribution for $0<\alpha <1$. From equation (\ref{eqn2.9}) in
Example \ref{examplestable}, the mapping $\Phi _{\mu }$ corresponds to the
upsilon transformation%
\begin{equation*}
\Upsilon _{\mu }\left( \eta \right) \left( A\right) =\int_{0}^{\infty
}\int_{0}^{\infty }1_{A}\left( xs\right) \eta \left( ds\right) \mu \left(
dx\right) \text{, \ \ }A\in \mathcal{B}\left( \mathbb{R}\backslash \left\{
0\right\} \right) \text{.}
\end{equation*}%
According to Theorem 3.4 in \cite{MR2421179}, the L\'{e}vy domain of $%
\Upsilon _{\mu }$ cannot be neither $\mathfrak{M}_{L}\left( \mathbb{R}%
^{+}\right) $ nor $\mathfrak{M}_{L}^{1}\left( \mathbb{R}^{+}\right) $
because $\mu $ has moments strictly smaller than $\alpha $. Nevertheless, by
equation (\ref{eqn2.9}), $\Upsilon _{\mu }\left( \eta \right) \in \mathfrak{M%
}_{L}\left( \mathbb{R}\right) $ if the image measure of $\eta $ and the
mapping $s\longmapsto s^{\alpha }$ belong to $\mathfrak{M}_{L}^{1}\left( 
\mathbb{R}^{+}\right) $. In such set of measures, $\Upsilon _{\mu }$ is
one-to-one. The cancellation property of $\mu $ is obtained in an analogous
way.
\end{example}

\section{On the law of a subordinated L\'{e}vy process}

In this section we study the law associated with a subordinated L\'{e}vy
process and we show that this class is in a bijection with the infinitely
divisible distributions asscociated with subordinators. For a fixed L\'{e}vy
process, as we will see below, the law of a subordinated L\'{e}vy process
appears as the restriction of the mapping $\Phi _{\mu }$ to the space of
infinitely divisible distributions with support on $\left( 0,\infty \right) $%
. In what follows, $\mathcal{L}\left( X\right) $ will denote the
distribution of the random variable $X$.

Let us start by introducing the mappings which we are interested in. Let $%
\mu _{L},\mu _{T}\in ID\left( \mathbb{R}\right) $ with characteristic
triplets $\left( \gamma ,b,\nu \right) $ and $\left( \beta _{0},0,\rho
\right) _{0}$, respectively. Put $\overline{\mu }\in ID\left( \mathbb{R}%
\right) $ as the probability distribution with characteristic triplet
(relative to $\tau $) $\left( \overline{\gamma },\overline{b},\overline{\nu }%
\right) $ given by 
\begin{eqnarray}
\overline{b} &=&b\beta _{0};  \notag \\
\overline{\gamma } &=&\gamma \beta _{0}+\int_{0}^{\infty }\int_{\left\vert
x\right\vert \leq 1}x\mu ^{s}\left( dx\right) \rho \left( ds\right) ;
\label{eqn3.2} \\
\overline{\nu }\left( dx\right) &=&\beta _{0}\nu \left( dx\right)
+\int_{0}^{\infty }\mu ^{s}\left( dx\right) \rho \left( ds\right) .  \notag
\end{eqnarray}%
Define $\Lambda _{\mu _{L}}:\mathcal{D}\left( \Lambda _{\mu _{L}}\right)
\rightarrow ID\left( \mathbb{R}\right) $, $\Lambda _{\mu _{T}}:\mathcal{D}%
\left( \Lambda _{\mu _{T}}\right) \rightarrow ID\left( \mathbb{R}\right) $
and $\Lambda :\mathcal{D}\left( \Lambda \right) \rightarrow ID\left( \mathbb{%
R}\right) $ by the formulae 
\begin{eqnarray*}
\Lambda _{\mu _{L}}\left( \mu _{T}\right) &=&\overline{\mu }; \\
\Lambda _{\mu _{T}}\left( \mu _{L}\right) &=&\overline{\mu }; \\
\Lambda \left( \mu _{L},\mu _{T}\right) &=&\overline{\mu },
\end{eqnarray*}%
for some domains $\mathcal{D}\left( \Lambda _{\mu _{L}}\right) ,\mathcal{D}%
\left( \Lambda _{\mu _{T}}\right) $ and $\mathcal{D}\left( \Lambda \right) .$
Observe that $\Lambda _{\mu _{L}}$ and $\Lambda _{\mu _{T}}$ are the
sections of $\Lambda $ for fixed $\mu _{L}$ and $\mu _{T},$ respectively.

In the following remarks $\overline{\mu }$ is the probability measure
associated with $\left( \overline{\gamma },\overline{b},\overline{\nu }%
\right) $ as in (\ref{eqn3.2}).

\begin{remark}
\label{Remark3.1} Observe that in general, $\left( \overline{\gamma },%
\overline{b},\overline{\nu }\right) $ is the characteristic triplet of an
infinitely divisible distribution if and only if $\beta _{0}\geq 0$ and $%
\rho \in \mathcal{D}_{L}\left( \Phi _{\mu _{L}}\right) $.
\end{remark}

\begin{remark}
\label{Remark3.2}Let $L$ be a L\'{e}vy process and $T$ a subordinator with
characteristic triplet $\left( \gamma ,b,\nu \right) $ and pair $\left(
\beta _{0},\rho \right) _{0}$, respectively. Consider $L_{T}$ to be the L%
\'{e}vy process obtained by subordinating $L$ through $T$. Then, by Theorem %
\ref{TheoremcqsubLB} and (\ref{eqn2.0})%
\begin{equation}
\overline{\mu }=\mathcal{L}\left( L_{T_{1}}\right) =\Phi _{\mu _{L}}\left(
\mu _{T}\right) \text{, \ \ }\mu _{T}\sim T_{1}\text{, }\mu _{L}\sim L_{1}%
\text{.}  \label{eqn3.2.1}
\end{equation}
\end{remark}

\begin{remark}
\label{Remark3.4}Observe that $\Lambda $ is not one-to-one in general, for
instance if $\mu _{T}\sim \left( a,0,0\right) $, $\mu _{L}\sim \left(
0,b,0\right) $, $\widetilde{\mu }_{T}\sim \left( a^{2},0,0\right) $ and $%
\widetilde{\mu }_{L}\sim \left( 0,\frac{1}{a}b,0\right) $. Then, by (\ref%
{eqn3.2}) 
\begin{equation*}
\Lambda \left( \mu _{L},\mu _{T}\right) =\Lambda \left( \widetilde{\mu }_{L},%
\widetilde{\mu }_{T}\right) .
\end{equation*}
\end{remark}

For the rest of this section, we study some properties of $\Lambda _{\mu
_{L}}$ and $\Lambda _{\mu _{T}}$, including their injectivity. Let $\mathcal{%
R}\left( \Lambda _{\mu _{L}}\right) $ and $\mathcal{R}\left( \Lambda _{\mu
_{T}}\right) $ be the range of $\Lambda _{\mu _{L}}$ and $\Lambda _{\mu
_{T}} $, respectively. Denote by $ID\left( \mathbb{R}^{+}\right) $ the
subset of infinitely divisible distributions with support on $\mathbb{R}^{+}$%
.

\begin{proposition}
\label{Proplinearlambda}We have that $\mathcal{D}\left( \Lambda _{\mu
_{L}}\right) =ID\left( \mathbb{R}^{+}\right) $ and $\mathcal{D}\left(
\Lambda _{\mu _{T}}\right) =ID\left( \mathbb{R}\right) $ and in general the
ranges $\mathcal{R}\left( \Lambda _{\mu _{L}}\right) $ and $\mathcal{R}%
\left( \Lambda _{\mu _{T}}\right) $ are proper subsets of $ID\left( \mathbb{R%
}\right) $. Furthermore, $\mathcal{R}\left( \Lambda _{\mu _{L}}\right) $ is
closed under convolutions and 
\begin{equation}
\Lambda _{\mu _{L}}\left( \mu ^{1}\ast \mu ^{2}\right) =\Lambda _{\mu
_{L}}\left( \mu ^{1}\right) \ast \Lambda _{\mu _{L}}\left( \mu ^{2}\right) ,%
\text{ \ \ }\mu ^{1},\mu ^{2}\in \mathcal{D}\left( \Lambda _{\mu
_{L}}\right) \text{.}  \label{eqn3.4}
\end{equation}%
In contrast, $\mathcal{R}\left( \Lambda _{\mu _{T}}\right) $ is not closed
under convolutions in general.
\end{proposition}

\begin{proof}
The first part is derived directly from the definition of $\Lambda _{\mu
_{L}}$ and $\Lambda _{\mu _{T}}$, Proposition \ref{popdomrange} and Remark %
\ref{Remark3.1}. Now, we proceed to verify (\ref{eqn3.4}). Let $\mu _{L}\in
ID\left( \mathbb{R}\right) .$ Suppose that $\mu \in \mathcal{D}\left(
\Lambda _{\mu _{L}}\right) $ and has pair $\left( \beta _{0},\rho \right)
_{0}$. Then by Remark \ref{Remark3.2} and Theorem 30.1 in \cite{sato}%
\begin{equation}
\log \widehat{\Lambda _{1}}\left( \mu \right) \left( \theta \right) =\psi
_{\mu }\left( \log \widehat{\mu _{L}}\left( \theta \right) \right) \text{, \ 
}\theta \in \mathbb{R}\text{,}  \label{eqn3.3}
\end{equation}%
where 
\begin{equation}
\psi _{\mu }\left( z\right) :=\log \int_{0}^{\infty }e^{zx}\mu \left(
dx\right) <\infty \text{, \ \ }z\in \mathbb{C}\text{,}\text{Re}\left(
z\right) \leq 0.  \label{eqn3.5}
\end{equation}%
Hence, if $\mu ^{1},\mu ^{2}\in \mathcal{D}\left( \Lambda _{\mu _{L}}\right)
,$ using that $\psi _{\mu ^{1}\ast \mu ^{2}}=\psi _{\mu ^{1}}+\psi _{\mu
^{2}}$, we get that for any\ $\theta \in \mathbb{R}$%
\begin{eqnarray*}
\log \widehat{\Lambda _{\mu _{L}}}\left( \mu ^{1}\ast \mu ^{2}\right) \left(
\theta \right) &=&\psi _{\mu ^{1}\ast \mu ^{2}}\left( \log \widehat{\mu _{L}}%
\left( \theta \right) \right) \\
&=&\psi _{\mu ^{1}}\left( \log \widehat{\mu _{L}}\left( \theta \right)
\right) +\psi _{\mu ^{2}}\left( \log \widehat{\mu _{L}}\left( \theta \right)
\right) \\
&=&\log \widehat{\overline{\mu }^{1}\ast \overline{\mu }^{2}}\left( \theta
\right) \text{,}
\end{eqnarray*}%
where $\Lambda _{\mu _{L}}\left( \mu ^{1}\right) =\overline{\mu }^{1}$ and $%
\Lambda _{\mu _{L}}\left( \mu ^{2}\right) =\overline{\mu }^{2},$ and $%
\widehat{\eta }$ stands for Fourier transform of the measure $\eta .$ This
shows (\ref{eqn3.4}) and in particular it follows that $\mathcal{R}\left(
\Lambda _{\mu _{L}}\right) $ is closed under convolutions$.$ The last part
is obtained by noting that in general for a fix $\mu _{T}\in ID\left( 
\mathbb{R}^{+}\right) $%
\begin{equation*}
\psi _{\mu _{T}}\left( \log \widehat{\mu ^{1}}\left( \theta \right) +\log 
\widehat{\mu ^{2}}\left( \theta \right) \right) \neq \psi _{\mu }\left( \log 
\widehat{\mu ^{1}}\left( \theta \right) \right) +\psi _{\mu }\left( \log 
\widehat{\mu ^{2}}\left( \theta \right) \right) \text{, \ \ }\theta \in 
\mathbb{R}\text{.}
\end{equation*}
\end{proof}

Note that in what follows, by \emph{continuity} we mean continuity with
respect to the weak convergence of probability measures.

\begin{proposition}
\label{Propcontinuitylambda}The mappings $\Lambda _{\mu _{L}}$ and $\Lambda
_{\mu _{T}}$ are continuous.
\end{proposition}

\begin{proof}
Suppose that $\mu _{L}\in ID\left( \mathbb{R}\right) $ and let $\mu \in 
\mathcal{D}\left( \Lambda _{\mu _{L}}\right) $. Then $\Lambda _{\mu
_{L}}\left( \mu \right) $ has characteristic exponent given by (\ref{eqn3.3}%
). Let $\left( \mu _{n}\right) _{n\geq 1}\subset \mathcal{D}\left( \Lambda
_{\mu _{L}}\right) $, such that $\mu _{n}\rightarrow \mu _{\infty }$ weakly.
From the previous proposition $\mu _{\infty }\in \mathcal{D}\left( \Lambda
_{\mu _{L}}\right) $. Thus, due to the Continuity Theorem, $\psi _{\mu
_{n}}\rightarrow \psi _{\mu _{\infty }}$ pointwise$,$ where $\psi _{\mu }$
is as in (\ref{eqn3.5}). The continuity of $\Lambda _{\mu _{L}}$ follows by
applying the previous observations to (\ref{eqn3.3}). Finally, by
interchanging the roles in the previous reasoning, it follows that $\Lambda
_{\mu _{T}}$ is continuous as well.
\end{proof}

From Propositions \ref{Proplinearlambda} and \ref{Propcontinuitylambda}, for
a given $\mu _{L}\in ID\left( \mathbb{R}\right) $, $\Lambda _{\mu _{L}}$ is
continuous and linear with respect to the convolution operation. Moreover,
Corollary \ref{injectprob} and (\ref{eqn3.2.1}) show that $\Lambda _{\mu
_{L}}$ is one-to-one. Hence, we have obtained the following result:

\begin{theorem}
\label{Theoremrecovery}Let $\mu _{L}\in ID\left( \mathbb{R}\right) $ be
given. If $\mu _{L}$ is not the Dirac's delta measure at zero, then $\Lambda
_{\mu _{L}}$ is a \textit{continuous group isomorphism}. Moreover, $\Lambda
_{\mu _{L}}\left( ID\left( \mathbb{R}\right) \right) \subsetneq ID\left( 
\mathbb{R}\right) $.
\end{theorem}

To conclude this part we present the injectivity of $\Lambda _{\mu _{T}}$.

\begin{theorem}
Let $\mu _{T}\in ID\left( \mathbb{R}^{+}\right) $ be given. If $\mu _{T}$ is
not the Dirac delta measure at zero, then $\Lambda _{\mu _{T}}$ is
one-to-one. Moreover, $\Lambda _{\mu _{T}}\left( ID\left( \mathbb{R}\right)
\right) \subsetneq ID\left( \mathbb{R}\right) $.
\end{theorem}

\begin{proof}
Assume that $\mu _{T}\in ID\left( \mathbb{R}^{+}\right) $ is not the Dirac
delta measure at zero and let $\mu ,\widetilde{\mu }\in ID\left( \mathbb{R}%
\right) $ such that $\Lambda _{\mu _{T}}\left( \mu \right) =\Lambda _{\mu
_{T}}\left( \widetilde{\mu }\right) $. Define by $L,\widetilde{L}$ and $T$
the L\'{e}vy processes associated with $\mu ,\widetilde{\mu }$ and $\mu _{T}$%
, respectively. Since $\mu _{T}\in ID\left( \mathbb{R}^{+}\right) $ and is
not the Dirac delta measure at zero, $T$ is a non-zero subordinator. Suppose
that the L\'{e}vy measure of $\mu _{T}$ is infinity in a neighborhood of
zero, then $T_{t}>0$ almost surely for any $t>0$. Consequently, by
independence, for all $t>0$%
\begin{eqnarray*}
\mathbb{E}\left( \left. e^{i\theta L_{T_{1}}}\right\vert T\right)
&=&e^{T_{1}\log \widehat{\mu }}; \\
\mathbb{E}\left( \left. e^{i\theta \widetilde{L}_{T_{1}}}\right\vert
T\right) &=&e^{T_{1}\log \widehat{\widetilde{\mu }}}.
\end{eqnarray*}%
But since $\Lambda _{\mu _{T}}\left( \mu \right) =\Lambda _{\mu _{T}}\left( 
\widetilde{\mu }\right) ,$ we have that almost surely 
\begin{equation}
e^{T_{1}\log \widehat{\mu }}=e^{T_{1}\log \widehat{\widetilde{\mu }}}\text{.}
\label{eqn3.6}
\end{equation}%
From this, it is immediate that $\mu =\widetilde{\mu }$. If the L\'{e}vy
measure of $\mu _{T}$ is finite, then the first jump time of $T$ is
finite almost surely. Following the same reasoning as before, we get that (%
\ref{eqn3.6}) holds for $T_{e_{1}}$ instead of $T_{1},$ where $e_{1}$ is the
first jump time of $T$. This concludes the proof.
\end{proof}

\section{On the recovery problem of time-changed infinitely divisible
processes}

In this section we study the recovery problem of subordinated L\'{e}vy bases
and $\mathcal{LSS}$ processes driven by a subordinated L\'{e}vy process.

Let $\left( L_{t}\right) _{t\geq 0}$ and $\left( T_{t}\right) _{t\geq 0}$ be
a L\'{e}vy process and an increasing process starting at zero and
independent of $L$ with c\`{a}dl\`{a}g paths, respectively. Consider $\left(
X_{t}\right) _{t\geq 0}$ given by 
\begin{equation}
X_{t}:=L_{T_{t}}\text{, \ \ }t\geq 0,  \label{eqn4.1}
\end{equation}%
which is the process obtained by time changing $L$ via $T$. These kind of
processes are of importance in mathematical finance since they are often
used to model the return process of a financial asset. In this context, when 
$L$ is a Brownian motion, $T$ plays the role of \textit{stochastic volatility%
}. Since in reality $T$ cannot be observed, a very important question
arises: Is there a way to \textit{recover} $T$ from $X$? More generally, is $%
T$ \textit{completely determined} by $X$, either in t\textit{he path-wise
sense} or \textit{in the distributional sense}? This issue is what we refer
to as the \textit{recovery problem} of time-changed L\'{e}vy processes.

Observe that when $L$ is a standard Brownian motion and $T$ is continuous
the recovery problem is trivial, because in this case $X$ is continuous and 
\begin{equation*}
T_{t}=\left[ X\right] _{t}\text{, \ \ a.s. for all }t>0\text{.}
\end{equation*}

\cite{WinMaT01} studied the problem when $L$ is a L\'{e}vy process (but not
compound Poisson) and $T$ being continuous. In this case $T$ can be
expressed in terms of the L\'{e}vy measure of $L$, the quadratic variation
of $X$ as well as its jumps.

\cite{GemMAdYor02} studied the conditional law of $X$ when $L$ is a Brownian
motion and $T$ is a surbordinator with pair $\left( \beta _{0},\rho \right)
_{0}$. They showed that 
\begin{equation*}
\mathbb{E}\left( \left. e^{-\lambda T_{t}}\right\vert \mathcal{F}_{t}\right)
=e^{-t\psi \left( \lambda \right) }M_{t}\left( \lambda ,\rho ,X\right) ,%
\text{ \ \ }\lambda ,t>0,
\end{equation*}%
where $\psi $ is the characteristic exponent of the Laplace transform of $T$%
, $\left( \mathcal{F}_{t}\right) _{t\geq 0}$ is the natural filtration of $X$
and $M_{t}\left( \lambda ,\rho ,X\right) $ is certain martingale depending
only on $\lambda ,$ $\rho $ and $\left( X_{s}\right) _{0\leq s\leq t}$. \cite%
{WinMaT01} improved this result to include any L\'{e}vy process for which
its associated law has density. In this case 
\begin{equation*}
\mathbb{E}\left( \left. e^{-\lambda T_{t}}\right\vert \mathcal{F}_{t}\right)
=e^{-t\psi \left( \lambda \right) }M_{t}\left( \lambda ,\rho ,X,\mu \right) ,%
\text{ \ \ }\lambda ,t>0,
\end{equation*}%
with $M_{t}\left( \lambda ,\rho ,X,\mu \right) $ certain martingale
depending only on $\lambda $, $\rho $, $\left( X_{s}\right) _{0\leq s\leq t}$
and $\mu \in ID\left( \mathbb{R}\right) ,$ the law of $L$ at time $1$.
Furthermore, the author shows that if $T$ is purely discontinuous, then $X$
and $T$ jumps at the same time, and%
\begin{equation*}
\mathbb{P}\left( \left. \Delta T_{t}\in dz\right\vert S=t,\Delta
X_{t}=y\right) =\frac{1}{c\left( \mu ,\rho \right) }f\left( y,z\right) \rho
\left( dz\right) ,
\end{equation*}%
where $c\left( \mu ,\rho \right) $ is a positive constant depending on the L%
\'{e}vy measure of $T$ and the law of $L$.

Nevertheless, in all the cases above, \textit{it is implicitly assumed that
the triplet of }$T$\textit{\ can be obtained by }$X$. What we have shown in
the previous section is that the triplet of $T$ is uniquely determined by $X$
for every non-zero L\'{e}vy process.

For the rest of the section, we study the distributional recovery problem
for subordinated L\'{e}vy bases and for the class of L\'{e}vy semistationary
processes.

\subsection{The recovery problem of subordinated L\'{e}vy bases}

In this part we study the recovery problem for subordinated L\'{e}vy bases.
It turned out that the non-homogeneous case is deducted directly from the
homogeneous case.

Recall that any arbitrary L\'{e}vy basis $T$ on $\mathbb{R}^{k}$ with
characteristic quadruplet $\left( \beta \left( s\right) ,b\left( s\right)
,\nu \left( s,\cdot \right) ,c\left( s\right) \right) $ is a measure-valued
field if and only if for $c$-almost all $s,$ $b\left( s\right) =0,$ $\nu
\left( s,\left( -\infty ,0\right) \right) =0$, the integral $\int_{\mathbb{R}%
^{+}}\left( 1\wedge x\right) \nu \left( s,dx\right) $ is finite and $\beta
\left( s\right) \geq \int_{0}^{1}x\nu \left( s,dx\right) $ (see Section 2).
This is equivalent to saying that the laws of the L\'{e}vy seeds $\mathbb{T=}%
\left( T^{\prime }\left( s\right) \right) _{s\in \mathbb{R}^{k}}$, say, $%
M=\left( \mu _{s}\right) _{s\in \mathbb{R}^{k}},$ have support on $\left[
0,\infty \right) $. Let $L$ be a homogeneous L\'{e}vy basis with
characteristic triplet $\left( \gamma ,b,\nu \right) $ and $T$ as before.
Put $L_{T}$ to be as in (\ref{eqn1.9.0}), i.e. $L_{T}$ is the L\'{e}vy basis
subordinated by the meta-time associated with $T$. Within this framework, $T$
is uniquely determined by $L_{T}$ in law, as the following theorem shows:

\begin{theorem}
\label{TheoremLBtimechange}Under the notation above, if $\mu \in ID\left( 
\mathbb{R}\right) $ has characteristic triplet $\left( \gamma ,b,\nu \right)
,$ and it is not the Dirac delta measure at zero, then the law of $L_{T}$ is
determined uniquely by the law of $T$, i.e. if there exists a non-negative L%
\'{e}vy basis $\widetilde{T}$, such that $L_{T}\overset{d}{=}L_{\widetilde{T}%
}$, then necessarily $T\overset{d}{=}\widetilde{T}$.
\end{theorem}

\begin{proof}
Let $T$ and $\widetilde{T}$ be two non-negative L\'{e}vy bases and $\left(
\beta \left( s\right) ,0,\nu \left( s,\cdot \right) ,c\left( s\right)
\right) $,$(\widetilde{\beta }\left( s\right) ,0,\widetilde{\nu }\left(
s,\cdot \right) ,\widetilde{c}\left( s\right) ),$ their corresponding
characteristic quadruplets. Denote by $M=\left( \mu _{s}\right) _{s\in 
\mathbb{R}^{k}}$ the collection of infinitely divisible laws related to the L%
\'{e}vy seeds of $T$, this is $\mu _{s}=\mathcal{L}\left( T^{\prime }\left(
s\right) \right).$ Define $\widetilde{M}=\left( \widetilde{\mu }_{s}\right)
_{s\in \mathbb{R}^{k}}$ in an analogous way. From the observation made at
the beginning of this subsection, we get that $M,\widetilde{M}\subset
ID\left( \mathbb{R}^{+}\right) .$ Now, suppose that $L_{T}\overset{d}{=}L_{%
\widetilde{T}}.$ Due to Theorem \ref{TheoremcqsubLB}, it follows that $c=%
\widetilde{c}$. Hence, $T\overset{d}{=}\widetilde{T}$ if and only if $\mu
_{s}=\widetilde{\mu }_{s}$ for $c$-almost all $s\in \mathbb{R}^{k}$. Let us
show that this holds whenever $L_{T}\overset{d}{=}L_{\widetilde{T}}$.

Since $L_{T}$ is a L\'{e}vy basis, there is an associated collection of
infinitely divisible laws $\left( \overline{\mu }_{s}\right) _{s\in \mathbb{R%
}^{k}}$, say, on $\mathbb{R}$, such that $\overline{\mu }_{s}=\mathcal{L}%
\left( L_{T}^{\prime }\left( s\right) \right) $, where $L_{T}^{\prime
}\left( s\right) $ is the infinitely divisible random variable with
characteristic triplet $\left( \overline{\gamma }\left( s\right) ,\overline{b%
}\left( s\right) ,\overline{\nu }\left( s,dx\right) \right) $, with $%
\overline{\gamma }\left( s\right) ,\overline{b}\left( s\right) $ and $%
\overline{\nu }\left( s,dx\right) $ as in Theorem \ref{TheoremcqsubLB}. An
analogous expression holds for $L_{\widetilde{T}}$ and its associated L\'{e}%
vy seeds $L_{\widetilde{T}}^{\prime }\left( s\right) $. By hypothesis $%
L_{T}^{\prime }\left( s\right) \overset{d}{=}L_{\widetilde{T}}^{\prime
}\left( s\right) $ for $c$-almost all $s\in \mathbb{R}^{k}.$ Moreover, from (%
\ref{eqn3.2}) and (\ref{eqn3.2.1}), for all $s\in \mathbb{R}^{k}$%
\begin{equation*}
\mathcal{L}\left( L_{T}^{\prime }\left( s\right) \right) =\Lambda _{\mu
}\left( \mu _{s}\right) =\Phi _{\mu }\left( \mu _{s}\right) \text{,}
\end{equation*}%
and 
\begin{equation*}
\mathcal{L}\left( L_{\widetilde{T}}^{\prime }\left( s\right) \right)
=\Lambda _{\mu }\left( \widetilde{\mu }_{s}\right) =\Phi _{\mu }\left( 
\widetilde{\mu }_{s}\right) .
\end{equation*}%
Therefore, for $c$-almost all $s\in \mathbb{R}^{k}$%
\begin{equation*}
\Lambda _{\mu }\left( \mu _{s}\right) =\Lambda _{\mu }\left( \widetilde{\mu }%
_{s}\right) \text{,}
\end{equation*}%
which, once one applies Theorem \ref{Theoremrecovery}, implies that $\mu
_{s}=\widetilde{\mu }_{s}$ for $c$-almost all $s\in \mathbb{R}^{k}$, just as
it was claimed.
\end{proof}

Due to the bijection between natural additive processes and L\'{e}vy bases
on $\mathbb{R}^{+}$, the previous theorem implies the following:

\begin{corollary}
Let $\left( L_{t}\right) _{t\geq 0}$ be a L\'{e}vy process and $\left(
T_{t}\right) _{t\geq 0}$ an increasing additive process. Then the law of $%
\left( T_{t}\right) _{t\geq 0}$ is completely determined by the law of the
time-changed process $\left( X_{t}\right) _{t\geq 0}$, where $X_{t}$ is as
in (\ref{eqn4.1}).
\end{corollary}

\begin{proof}
It is enough to point out that in this case, by Theorem 4.2 in \cite%
{LevyItodecompositionrm} and Example 3.3 in \cite{BNJP(12)}, there exist a
homogeneous L\'{e}vy basis $\mathbf{L}$ and a non-negative L\'{e}vy basis $%
\mathbf{T}$, such that 
\begin{equation*}
X_{t}=\mathbf{L}_{\mathbf{T}}\left( \left[ 0,t\right] \right) ,\text{ \ \
a.s. for all }t>0,
\end{equation*}%
where $\mathbf{L}_{\mathbf{T}}$ is the time-changed L\'{e}vy basis via $%
\mathbf{T}$. Thus, the result follows from the previous theorem.
\end{proof}

\begin{remark}
Note that by reasoning as in Theorem \ref{TheoremLBtimechange}, if $\left(
T_{t}\right) _{t\geq 0}$ is an arbitrary infinitely divisible chronometer
(not necessarily with independent increments), we have that for all $t>0$%
\begin{equation*}
\mathcal{L}\left( X_{t}\right) =\Lambda _{\mu }\left( \mu _{t}\right) ,
\end{equation*}%
where $\mu _{t}$ is the law of $T$ at time $t$. Hence, if there exists
another infinitely divisible chronometer $(\widetilde{T}_{t}) _{t\geq 0}$
such that $\left( L_{T_{t}}\right) _{t\geq 0}\overset{d}{=} (L_{\widetilde{T}%
_{t}}) _{t\geq 0}$, then necessarily for every $t>0$%
\begin{equation*}
T_{t}\overset{d}{=}\widetilde{T}_{t}\text{,}
\end{equation*}%
which in general does not imply that $T\overset{d}{=}\widetilde{T}$.
Therefore, more general mappings than $\Phi _{\mu }$ should be considered,
namely transformations on $ID\left( \mathbb{R}^{\left( 0,\infty \right)
}\right) $, the space of infinitely divisible laws on $\mathbb{R}^{\left(
0,\infty \right) }$, into itself. See \cite{sdfields} for more details. We
leave such a problem for future research.
\end{remark}

\subsection{L\'{e}vy semistationary processes and the recovery problem}

Consider the subclass of L\'{e}vy semistationary processes given by the
formula,%
\begin{equation}
Y_{t}:=\int_{-\infty }^{t}f\left( t-s\right) \sigma _{s}dL_{s}^{T}\text{,\ \
\ }t\in \mathbb{R}\text{.}  \label{eqn5.1}
\end{equation}%
where $f$ is a deterministic function such that $f\left( x\right) =0$ for $%
x\leq 0$, $\sigma $ a c\`{a}dl\`{a}g predictable process and $X^{T}$ is a
two-sided L\'{e}vy process for which%
\begin{equation*}
L_{s}^{T}:=L_{T_{s}}\text{, \ \ }s\geq 0,
\end{equation*}%
with $L$ a L\'{e}vy process with triplet $\left( \gamma ,b,\nu \right) $ and 
$T$ a subordinator with couple $\left( \beta _{0},\rho \right) _{0}$. Note
that the process $\sigma $ introduces stochastic volatility through \textit{%
stochastic amplitude modulation} whereas $T$ introduces stochastic
volatility through \textit{stochastic intensity modulation}. Thus, $\sigma $
changes the height of the jumps and variation of $L$ while $T$ randomizes
the jump times. Observe that $Y$ also can be written as 
\begin{equation*}
Y_{t}=\int_{-\infty }^{t}f\left( t-s\right) dX_{s},\ \ \ t\in \mathbb{R},
\end{equation*}%
where $X$ is the increment semimartingale given by 
\begin{equation*}
X_{t}-X_{s}=\int_{s}^{t}\sigma _{r}dL_{r}^{T},\ \ \ s<t.
\end{equation*}%
Therefore, in this case we may think in terms of two different kinds of
recovery problems. The first, involving $\sigma $ and the second related to $%
T$. If $T_{t}=t$ and $L$ is a Brownian motion then $\sigma $ can be
recovered pathwise. See Theorem 3.1 in \cite%
{RePEc:eee:spapps:v:123:y:2013:i:7:p:2552-2574}. More generally, if $Y$ is 
\textit{invertible}, i.e. $X_{t}$ can be obtained by limit of linear
combinations of $Y$, hence $\sigma $ can be recovered pathwise, because 
\begin{equation*}
\left[ X\right] _{t}=\int_{0}^{t}\sigma _{s}^{2}d\left[ L\right] _{s}^{T}%
\text{, \ \ a.s. }t\geq 0\text{.}
\end{equation*}%
Note that in this general case, the law of $Y$ is no longer infinitely
divisible. Moreover, it is not clear whether $T$ can be identified jointly
with $\sigma $.

Let us consider some examples where the invertibility of $Y$ plays an important
role.

\begin{example}[Ornstein-Uhlenbeck processes]
Put $f\left( x\right) =e^{-x}$ for $x\geq 0$ and $\sigma \equiv 1$. Then $Y$
is an Ornstein-Uhlenbeck process satisfying the Langevin equation 
\begin{equation*}
dY_{t}=Y_{t}dt+dL_{t}^{T}\text{, \ \ }t\geq 0.
\end{equation*}%
This implies that the law of $Y$ is uniquely determined by the law of $L^{T}$%
, but since the law of $L^{T}$ is uniquely determined by the law of $T$ (for
a fixed $L$)$,$ we conclude that the law of $Y$ is uniquely determined by
the one of $T$. Here we can recover (in the distributional sense) $T$ via $Y$%
.
\end{example}

\begin{example}[Semimartingale case]
It is well known that the process $Y$ is not in general a semimartingale.
Suppose that $L$ is a square integrable martingale and $Y$ is a
semimartingale. Then $Y$ admits the following representation (see \cite%
{RePEc:aah:create:2010-18} and \cite{ComteRenault1996}) 
\begin{equation*}
Y_{t}=Y_{0}+\int_{0}^{t}A_{s}ds+f\left( 0+\right) \int_{0}^{t}\sigma
_{s}dL_{T_{s}}\text{, \ \ a.s for }t\geq 0,
\end{equation*}%
where 
\begin{equation*}
A_{s}:=\int_{-\infty }^{s}f^{\prime }\left( s-r\right) \sigma _{r}dL_{T_{r}}%
\text{, \ \ }s\in \mathbb{R}\text{.}
\end{equation*}%
for a suitable function $f$. In this case $Y$ is invertible and $\sigma $
can be recovered pathwise.
\end{example}

The previous examples showed the importance of the invertibility in the
recovery problem for $\mathcal{LSS}$ processes. This means that in the
invertible case, the recovery problem for $Y$ is analogous to the one of $X$%
. When $\sigma \equiv 1$, $X$ is just the subordinated L\'{e}vy process by $%
T $ and in this case the invertibility of $Y$ holds under very weak
conditions.

\begin{theorem}
Suppose that $Y$ is given as in (\ref{eqn5.1}) with $\sigma \equiv 1$.
Assume in addition that $f$ is integrable with non-vanishing Fourier
transform. Then, for a given L\'{e}vy processes $L$, the law of $T$ is
completely determined by the law of $Y$.
\end{theorem}

\begin{proof}
Let $L$ a fixed L\'{e}vy process with associated law $\mu \in ID\left( 
\mathbb{R}\right) $. From Theorem \ref{injectprob} 
\begin{equation}
\mathcal{L}\left( T_{t}\right) =\Lambda _{\mu }^{-1}\left[ \mathcal{L}\left(
L_{T_{t}}\right) \right] \text{, \ \ }t\in \mathbb{R}\text{.}  \label{eqn5.2}
\end{equation}%
Moreover, thanks to Theorem 6 in \cite{Sau14}, $Y$ is invertible. Thus, for
every $t\in \mathbb{R}$, there are two sequences $\left( \theta
_{j}^{t}\right) _{j=1}^{n}$ and $\left( s_{j}^{t}\right) _{j=1}^{n}$, such
that 
\begin{equation*}
\mathcal{L}\left( L_{T_{t}}\right) =\lim_{n\rightarrow \infty }\mathcal{L}%
\left( \sum_{j=1}^{n}\theta _{j}^{t}Y_{s_{j}^{t}}\right) ,
\end{equation*}%
where the limit is taken in the weak sense. Since the law of $T$ is
completely determined by the law of $T_{t}$ for some $t$, the result follows
by inserting the previous equation into (\ref{eqn5.2}).
\end{proof}

\begin{remark}
Observe that in general the law of $Y$ is not uniquely determined by the law
of $T$. If this was true, it would imply that the mapping $\mu \mapsto 
\mathcal{L}\left( \int_{\mathbb{R}}f\left( s\right) dL_{s}^{\mu }\right) $,
where $\mu \in ID\left( \mathbb{R}\right) $ and $L^{\mu }$ is the L\'{e}vy
process associated with $\mu $, is one-to-one, which is not true in general.
\end{remark}

\begin{example}[Gamma kernel]
Let $f\left( x\right) =e^{-x}x^{\alpha }$ for $x\geq 0$ and $\alpha >-1$. We
have that the corresponding $\mathcal{LSS}$ process is invertible as Theorem
6 in \cite{Sau14} shows. When $\alpha =0$, $Y$ is just an Ornstein-Uhlenbeck
process and in this case the mapping $\mu \mapsto \mathcal{L}\left(
\int_{0}^{\infty }e^{-s}dL_{s}^{\mu }\right) $, where $\mu \in ID\left( 
\mathbb{R}\right) $ and $L^{\mu }$ is the L\'{e}vy process associated with $%
\mu $, is in a bijection with the subclass of selfdecomposable
distributions. Denoting by $\Psi ^{0}$ such a mapping we see that%
\begin{equation*}
\mathcal{L}\left( \int_{0}^{\infty }e^{-s}dL_{T_{s}}\right) =\Psi ^{0}\circ
\Lambda \left( \mu _{L},\mu _{T}\right) ,
\end{equation*}%
where $\mu _{L}\sim L_{1}$, $\mu _{T}\sim T_{1}$, with $L$ and $T$ a L\'{e}%
vy process and $T$ a subordinator, respectively. In particular, for a fixed L%
\'{e}vy process the mapping $\Psi ^{0}\circ \Lambda _{\mu _{L}}$ is a
bijection between the subclass of $ID\left( \mathbb{R}^{+}\right) $ for
which its L\'{e}vy measure has log-moments outside of zero, and the class of
selfdecomposable distributions. In this case, the law of $Y$ is uniquely
determined by the law of $T$ and viceversa.\newline
On the other hand, when $-1<\alpha <0$, the associated $\mathcal{LSS}$
process is no longer a semimartingale and it is not bounded with positive
probability. Moreover, the mapping $\mu \mapsto \mathcal{L}\left(
\int_{0}^{\infty }e^{-s}s^{\alpha }dL_{s}^{\mu }\right) $ maps a subclass of 
$ID\left( \mathbb{R}\right) $ into a proper subset of selfdecomposable
distributions. As before, let $\Psi ^{\alpha }$ denote such a
transformation, then for a fixed L\'{e}vy process $L$ for which $\mu _{L}\in
ID\left( \mathbb{R}\right) $ is its law at time $1$, we have that%
\begin{equation*}
\mathcal{L}\left( \int_{0}^{\infty }e^{-s}s^{\alpha }dL_{T_{s}}\right) =\Psi
^{\alpha }\circ \Lambda _{\mu _{L}}\left( \mu _{T}\right) .
\end{equation*}%
In \cite{onlssgamma}, it was shown that under some conditions, $\Psi
^{\alpha }$ is one-to-one. Therefore, for every $-1<\alpha \leq 0$, the
marginal law of the related $\mathcal{LSS}$ process is uniquely determined
by the law of the time change $T$.
\end{example}

\subsubsection*{Acknowledgement}

We wish to thank Ole E. Barndorff-Nielsen and Victor M.~P\'{e}rez Abreu
C.~for helpful comments on an earlier draft of this article. This work was
carried out while the first author visited Imperial College London. He
gratefully acknowledges the hospitality and financial support during this
research visit. A.~E.~D.~Veraart acknowledges financial support by a Marie
Curie FP7 Integration Grant within the 7th European Union Framework
Programme.

\bibliographystyle{apa}
\bibliography{bibliografia}

\end{document}